\documentclass[a4paper, 11pt, headings = small, abstract]{scrartcl}
\usepackage[english]{babel}

\usepackage{amscd,amsmath,amsthm,array,hhline,mathrsfs, enumerate, booktabs,fancybox,calc,textcomp,xcolor,graphicx,bbm,xspace,nicefrac,stmaryrd,url,arcs,listings,pifont}

\usepackage[theoremfont]{newpxtext}
\usepackage[vvarbb, upint, bigdelims]{newpxmath}
\usepackage[scr=boondoxupr, scrscaled = 1.05, cal = euler, calscaled =1.05]{mathalpha}
\usepackage[scaled=0.95]{inconsolata}
\usepackage{accents}
\usepackage{dsfont}

\DeclareMathAlphabet{\mathsf}{OT1}{\sfdefault}{m}{n}
\SetMathAlphabet{\mathsf}{bold}{OT1}{\sfdefault}{b}{n}
\DeclareMathAlphabet{\mathfrak}{U}{jkpmia}{m}{it}
\SetMathAlphabet{\mathfrak}{bold}{U}{jkpmia}{bx}{it}

\DeclareMathAlphabet{\mathsf}{OT1}{\sfdefault}{m}{n}
\SetMathAlphabet{\mathsf}{bold}{OT1}{\sfdefault}{b}{n}
\DeclareMathAlphabet{\mathfrak}{U}{jkpmia}{m}{it}
\SetMathAlphabet{\mathfrak}{bold}{U}{jkpmia}{bx}{it}
\usepackage[utf8]{inputenc}

\usepackage{enumitem}
\usepackage{etoolbox}
\usepackage[normalem]{ulem}
\usepackage{mathtools}
\usepackage{geometry}
\usepackage{float}

\usepackage{bm}
\usepackage{tikz}
\usepackage[outdir =./]{epstopdf}
\usepackage[citestyle = numeric-comp,bibstyle = numeric, giveninits=true, natbib = true, backend = bibtex,maxbibnames=99, url=false]{biblatex}
\numberwithin{equation}{section}

\usepackage{chngcntr}
\counterwithin{figure}{section}

\usepackage{xcolor}
\usepackage[colorlinks=true, allcolors = myteal]{hyperref}

\usepackage{verbatim}

\definecolor{WIMgreen}{RGB}{60 134 132}
\definecolor{UMblue}{RGB}{4 47 86}
\definecolor{myteal}{RGB}{0 123 137}
\definecolor{material_green}{RGB}{27 43 52}
\definecolor{dracula_pink}{RGB}{180 93 149}
\definecolor{dracula_blue}{RGB}{40 42 54}
\definecolor{dracula_turq}{RGB}{92 143 159}
\definecolor{dracula_orange}{RGB}{255 184 108}
\definecolor{material_petrol}{RGB}{2 119 189}
\definecolor{Purple}{RGB}{103 58 183}
\definecolor{cs}{rgb}{0.0, 0.44, 1.0}
\definecolor{oucrimsonred}{rgb}{0.6, 0.0, 0.0}

\theoremstyle{plain}
\newtheorem{theorem}{Theorem}[section]
\newtheorem*{theorem*}{Theorem}
\newtheorem{proposition}[theorem]{Proposition}
\newtheorem{lemma}[theorem]{Lemma}
\newtheorem{corollary}[theorem]{Corollary}

\theoremstyle{definition}

\theoremstyle{assumption}

\theoremstyle{remark}
\newtheorem{remark}[theorem]{Remark}
\newtheorem{example}[theorem]{Example}
\setkomafont{sectioning}{\normalcolor\bfseries}
\setkomafont{descriptionlabel}{\normalcolor\bfseries}
\setkomafont{section}{\Large}
\setkomafont{subsection}{\large}
\setkomafont{subsubsection}{\large}
\setkomafont{paragraph}{\large}
\setkomafont{subparagraph}{\large}
\setkomafont{author}{\large}
\setkomafont{date}{\normalsize}

\def\E{\mathbb{E}}

\def\N{\mathbb{N}}
\def\P{\mathbf{P}}
\def\N{\mathbb{N}}

\def\R{\mathbb{R}}
\definecolor{darkred}{rgb}{0,0.6,0}
\def\Z{\mathbb{Z}}

\def\Op{\overline{\bm\Psi}}

\newcommand{\cF}{\mathcal{F}}

\newcommand{\ep}{\varepsilon}

\newcommand{\PP}{\mathbb{P}}

\newcommand{\X}{\mathbf{X}}

\renewcommand{\hat}{\widehat}
\newcommand{\e}{\mathrm{e}}
\renewcommand{\tilde}{\widetilde}%

\renewcommand{\d}{\mathop{}\!\mathrm{d}}
\newcommand{\lebesgue}{\bm{\lambda}}
\newcommand{\overbar}[1]{\mkern 1.5mu\overline{\mkern-1.5mu#1\mkern-1.5mu}\mkern 1.5mu}

\restylefloat{table}

\newcommand*\diff{\mathop{}\!\mathrm{d}}

\newcommand{\one}{\mathbf{1}}
\newcommand*{\Cdot}{\raisebox{-0.25ex}{\scalebox{1.2}{$\cdot$}}}

\newcommand{\vertiii}[1]{{\left\vert\kern-0.25ex\left\vert\kern-0.25ex\left\vert #1
\right\vert\kern-0.25ex\right\vert\kern-0.25ex\right\vert}}

\let\originalleft\left
\let\originalright\right
\renewcommand{\left}{\mathopen{}\mathclose\bgroup\originalleft}
\renewcommand{\right}{\aftergroup\egroup\originalright}

\newlist{todolist}{itemize}{2}
\setlist[todolist]{label=$\square$}

\makeatletter
\newcommand{\specificthanks}[1]{\@fnsymbol{#1}}
\makeatother

\bibliography{levy_kontrolle}

\allowdisplaybreaks

\makeatother
\title{\fontsize{16}{19} \selectfont Concentration analysis of multivariate elliptic diffusion processes}
\author{Cathrine Aeckerle-Willems\thanks{University of Mannheim, Department of Economics, L7 3--5, 68161 Mannheim, Germany. \newline Email: \href{mailto:aeckerle@uni-mannheim.de}{aeckerle@uni-mannheim.de}} \and Claudia Strauch\thanks{Aarhus University, Department of Mathematics, Ny Munkegade 118, 8000 Aarhus C, Denmark \newline Email: \href{mailto:strauch@math.au.dk}{strauch@math.au.dk}/\href{mailto:trottner@math.au.dk}{trottner@math.au.dk}} \and Lukas Trottner\footnotemark[2]}
\date{\vspace{-5ex}}
\begin{document}
\maketitle

\begin{abstract}
We prove concentration inequalities and associated PAC bounds  for continuous- and discrete-time additive functionals for possibly unbounded functions of multivariate, nonreversible diffusion processes. Our analysis relies on an approach via the Poisson equation allowing us to consider a very broad class of subexponentially ergodic processes. 
These results add to existing concentration inequalities for additive functionals of diffusion processes which have so far been only available for either bounded functions or for unbounded functions of processes from a significantly smaller class.
We demonstrate the power of these exponential inequalities by two examples of very different areas. Considering a possibly high-dimensional parametric nonlinear drift model under sparsity constraints, we apply the continuous-time concentration results to validate the restricted eigenvalue condition for Lasso estimation, which is fundamental for the derivation of oracle inequalities.
The results for discrete additive functionals are used to investigate the unadjusted Langevin MCMC algorithm for sampling of moderately heavy-tailed densities $\pi$. In particular, we provide PAC bounds for the sample Monte Carlo estimator of integrals $\pi(f)$ for polynomially growing functions $f$ that quantify sufficient  sample and step sizes for approximation within a prescribed margin with high probability.
\end{abstract}

\section{Introduction}

Concentration inequalities for additive functionals belong to the fundamental probabilistic tools in statistics and related areas such as statistical learning and reinforcement learning since they allow exact quantification of the deviation of estimators from a given target. In particular, concentration inequalities for independent data such as Hoeffding, Bernstein and McDiarmid inequalities are of central importance for deriving PAC guarantees in classification and regression contexts (see, e.g., \cite{devroye96,wain19}). 
While such questions have been well understood for decades in classical settings for independent or strongly mixing data---see also the recent investigations of Bernstein and Hoeffding inequalities and related applications in statistical learning for Markov chains with spectral gap in \cite{jiang18,fan21}---the general picture for additive functionals of diffusion processes is less clear. Particularly when it comes to unbounded functionals, whose deviation properties around their ergodic mean are fundamentally important in a multitude of applications, useful results are rather scarce. Important achievements in this direction can be found in \cite{cattiaux05, gao14}, where for a restricted class of reversible diffusion processes exponential inequalities are derived by means of functional inequalities. While these results are mathematically elegant and explicitly quantify the contribution of the asymptotic variance, they come at the price of structural constraints on the diffusion coefficients which are hard to verify and often inappropriate for specific applications. 

The goal of this paper is therefore to derive usable exponential concentration inequalities for unbounded functionals, both for continuous as well as discrete multivariate diffusion data, under comparatively weak assumptions on the coefficients and the speed of ergodicity. With our particular focus on applications, we translate these inequalities into PAC bounds for the approximation task and demonstrate their usefulness in  specific high-dimensional applications to (i) penalized drift estimation under sparsity constraints, where we extend results for the classical Ornstein--Uhlenbeck model in \cite{gaiffas19,ciolek20} to more flexible parametrized models with relaxed ergodicity assumptions, and (ii) performance guarantees for unadjusted Langevin MCMC algorithms for heavy-tailed target sampling, which is a setting that substantially differs from the related pioneering work \cite{dala17,durmus17} for strongly log-concave targets. Here, for a given quantity of interest $\pi$ and a sample based estimator $\hat{\pi}_t$ with $t \in \mathbb{T}$ ---where  $\mathbb{T} = [0,\infty)$ or $\mathbb{T} = \Delta \N_0$ for some sampling distance $\Delta > 0$, depending on whether continuous or discrete data is available---, we say that $\hat{\pi}_t$ satisfies an $(\varepsilon,\delta)$-PAC bound for $t \geq T(\varepsilon,\delta) \in \mathbb{T}$, given $\varepsilon > 0, \delta \in (0,1)$, if  
\[
\forall t \geq T(\varepsilon,\delta),
\quad
\PP\big(\lvert \hat{\pi}_t - \pi \rvert \leq \varepsilon \big) \geq 1- \delta, \]
i.e., given a sample length of at least $T(\varepsilon,\delta)$, $\hat{\pi}_t$ approximates the target $\pi$ within an $\varepsilon$-margin with probability at least $1-\delta$. Such results are statistically much more insightful than  upper bounds on the mean deviation, which do not reveal detailed information on the distribution of the loss.

In our particular context, the objectives are exponential inequalities and associated PAC bounds of sample mean estimators of the quantity $\pi = \mu(f)\coloneqq \int f(x) \, \mu(\diff{x}) $, where $\mu$ is the stationary distribution of a subexponentially ergodic elliptic diffusion $\X$ and $f$ is a polynomially growing function. That is, we provide an in-depth analysis of the deviations around $\pi$ of $\hat{\pi}_t = t^{-1/2} \mathbb{G}_t(f)$, where
\begin{equation}\label{eq:add_cont}
\mathbb{G}_t(f) \coloneqq \frac{1}{\sqrt{t}} \int_0^t f(X_s) \diff{s},
\end{equation}
given continuous data $(X_s)_{0 \leq s \leq t}$, and of $\hat{\pi}_{n\Delta} = (n\Delta)^{-1/2}\mathbb{G}_{n,\Delta}(f)$, where
\begin{equation}\label{eq:add_disc}
\mathbb{G}_{n,\Delta}(f) \coloneqq \frac{1}{\sqrt{n \Delta}} \sum_{k = 1}^n f(X_{k \Delta}) \Delta,
\end{equation}
given discrete data $(X_{k\Delta})_{k=1,\ldots,n}$, as well as their burned-in versions. Since our specific framework is what sets this paper apart from related studies such as \cite{gao14}, we will now introduce both the class of processes we are working with as well as the \textit{Poisson equation} and its solution studied in \cite{pardoux01}, which is at the heart of our theoretical analysis based on \textit{martingale approximation}.

\paragraph{Basic framework}
Consider a $d$-dimensional elliptic diffusion that is given as the weak solution to the SDE
\begin{equation}\label{eq:sde}
\diff{X_t} = b(X_t) \diff{t} + \sigma(X_t) \diff{W_t},
\end{equation}
where $b \colon \R^d \to \R^d$ is a locally Lipschitz drift vector such that $\lVert b(x) \rVert \lesssim 1 + \lVert x \rVert^{q^\prime}$ for some $q^\prime \geq 0$ and $\sigma \in \R^{d\times d}$ is a uniformly continuous, bounded and locally Lipschitz $d \times d$-matrix-valued function such that $a \coloneqq \sigma\sigma^\top$ is uniformly elliptic, i.e.,
\[\langle a(x) \eta\slash \lVert \eta \rVert, \eta\slash \lVert \eta \rVert \rangle \geq \lambda, \quad x \in \R^d, \eta \in \R^d \setminus\{0\},\]
for some constant $\lambda > 0$.
We denote by $(\X,(\PP^x)_{x \in \R^d})$ the Markovian weak solution of \eqref{eq:sde} such that under $\PP^x$ the process  $\X$ solves \eqref{eq:sde} with initial condition $X_0 = x$  and has continuous paths almost surely. 
Note that $\X$ has the Feller property, cf.\ \cite[Corollary 11.1.5]{stroock06}, and is therefore Borel right such that it falls into the general framework for stability of Markov processes.
Without loss of generality, we may assume that there exists a family of shift operators $(\theta_t)_{t \geq 0}$ for $\X$, that is, $X_t \circ \theta_s = X_{t+s}$ for any $s,t \geq 0$. Let $\lambda_-,\lambda_+, \Lambda$ be the tightest constants such that, for any $x \neq 0$,
\[0< \lambda_- \leq \langle a(x) x \slash \lVert x \rVert, x/ \lVert x \rVert \rangle \leq \lambda_+, \quad \mathrm{tr}(a(x))/d \leq \Lambda,\]
where our assumptions guarantee that such constants always exist since we may always choose  $\Lambda = d^{-1} \sup_{x \in \R^d} \mathrm{tr}(a(x)) < \infty$, $\lambda_- = \lambda$ and $\lambda_+ = \sup_{x \in \R^d} \lVert \sigma(x) \rVert^2 < \infty$.

Our subsequent analysis substantially relies on the following growth condition on the drift, 
\begin{enumerate}[label = ($\mathscr{A}(q)$), ref =($\mathscr{A}(q)$), leftmargin = *]
\item if $\lVert x \rVert \geq M_0$, then $\langle b(x), x/\lVert x \rVert \rangle \leq -\mathfrak{r} \lVert x \rVert^{-q},$ \label{cond:drift}
\end{enumerate}
where $q \in [-1,1), M_0 \geq 0, \mathfrak{r} > 0$. 
For $q=0$, this condition equals the standard ergodicity condition in many recent investigations of multivariate diffusion processes exploiting the exponential $\beta$-mixing property. As will be discussed in Section \ref{sec:basics}, the case $q > 0$ corresponds to a subexponential ergodic behaviour of the diffusion.

Our approach to deviation inequalities is driven by the martingale approximation technique, which has been employed for the same purpose in the literature under more restrictive structural assumptions. \cite{aeckerle21} study concentration inequalities in the context of scalar exponentially ergodic diffusions in the regime $q = 0$ with polynomially growing drift, and in \cite{nickl20}, multivariate diffusions with unit diffusion matrix and periodic Lipschitz drift are considered. 
\cite{galt07} essentially treat the scalar dissipative case with $q = -1$. 
All of these papers put a special emphasis on uniformicity of the concentration inequalities with respect to the diffusion coefficients in order to apply them to statistical minimax estimation problems. Moreover, the martingale approximation is employed in \cite{mattingly10} for providing $L^2$ convergence guarantees of Monte Carlo estimators for well-behaved SDEs on the torus based on samples obtained by numerical  approximation schemes.

Central to the martingale approximation technique is the existence of a solution to the Poisson equation 
\begin{equation}\label{eq:poisson}
Lu = f,
\end{equation}
for appropriate functions $f$ where, given $u \in L^{1}_{\mathrm{loc}}(\R^d)$ having weak partial derivatives up to second order belonging to $L^{1}_{\mathrm{loc}}(\R^d)$,
\[Lu(x) = \langle b(x), \nabla u(x) \rangle + \frac{1}{2} \mathrm{tr}\big(a(x)D^2u(x)\big), \quad x \in \R^d,\]
is a second order local operator. 
Note that, on the domain $\mathcal{C}^2_0(\R^d)$, $L$ is the infinitesimal generator of the diffusion process. In the scalar case, \eqref{eq:poisson} has an explicit $\mathcal{C}^2$-solution, which is used in \cite{aeckerle21} to obtain $\sup$-norm moment bounds for empirical processes that are uniform over a class of SDE coefficients. Such results can then be employed for minimax optimal $\sup$-norm adaptive drift estimation as demonstrated in \cite{aeckerle18b}.

For multivariate diffusions, such explicit solutions are not obtainable in general such that one needs to deal with the Poisson equation in a more abstract manner. In \cite{pardoux01}, the authors demonstrate that in our framework,  for any $f\colon \R^d \to \R$ such that $\lvert f(x) \rvert \leq \mathfrak{L}(1+ \lVert x \rVert^\eta)$ for some finite constants $\mathfrak{L} > 0, \eta \geq 0$, there exists a solution $u[f] \in \bigcap_{p > 1} \mathcal{W}^{2,p}_{\mathrm{loc}}(\R^d)$ that is unique in the local Sobolev space $\mathcal{W}^{2,p}_{\mathrm{loc}}(\R^d)$ for any $p > d$.
This solution is given as 
\[u[f](x) = \int_0^\infty \E^x[-f(X_t)] \diff{t}, \quad x \in \R^d,\]
i.e., $u[f](x)$ is expressed as the potential of $-f$ under $\PP^x$. 
Therefore, for such $f$ we denote 
\[
L^{-1}[f](x) \coloneqq \int_0^\infty \E^x[-f(X_t)] \diff{t}, \quad x \in \R^d,
\]
such that $LL^{-1}[f] = f$, $\lebesgue$-a.e., where $\lebesgue$ denotes the Lebesgue measure on $\R^d$.  
The Sobolev regularity of $L^{-1}[f]$ is an essential property for our purposes, since it allows us to apply the It\={o}--Krylov formula for martingale approximation. 
This approach will enable us to conclude the desired deviation inequalities from moment bounds for the martingale approximation.

\paragraph{Outline and main results} In Section \ref{sec:basics}, we  collect some essential known facts on the subexponentially ergodic nature of the diffusion $\X$ implied by the drift condition \ref{cond:drift} and put them into a form suited to our needs. In Section \ref{sec:conti}, we present our first main result, the concentration inequality for the continuous-time scaled additive functional $\mathbb{G}_t(f)$ for polynomially growing $f$ (Theorem \ref{prop:conc}) which is based on our derivation of the martingale approximation $\mathbb{G}_t(f)$ and bounds on the solution to the Poisson equation and its gradient going back to \cite{pardoux01}. We translate these inequalities into stationary and non-stationary PAC bounds in Corollary \ref{coro:pac} and \ref{coro:burnin}, respectively. In Section \ref{sec:discrete}, we then proceed to derive explicit deviation bounds in terms of the sampling frequency $\Delta$ and number of observations $n$ for the discrete scaled additive functional $\mathbb{G}_{n,\Delta}(f)$ by combining Theorem \ref{prop:conc} with an approximation argument, see Theorem \ref{theo:discrete_conc}. As for the continuous data, we use this result to infer PAC bounds for the sample mean estimator and its burn-in version. In Section \ref{sec:lasso},
we apply the continuous-time results to the problem of estimating the coefficients in a possibly high-dimensional drift model of the form $b_{\theta_0} = \sum_{j=1}^N \theta_j\psi_j$, $\theta_0=(\theta_1,...,\theta_N)\in\R^N$,
given a dictionary $(\psi_j)_{1\leq j\leq N}$ of Lipschitz continuous functions $\psi_j\colon\R^d\to \R^d$ under sparsity constraints on the coefficients via a Lasso approach. Our concentration inequality is the key to showing that the central restricted eigenvalue condition is in place, which then in turn yields oracle inequalities in line with those well known in the classical regression context. Finally, Section \ref{sec:mcmc} is devoted to an application of our discrete deviation results, where we study the convergence properties of the unadjusted Langevin algorithm for moderately heavy-tailed target distributions $\pi$, in terms of sufficient sample and step size conditions for sampling within an $\varepsilon$-margin in total variation as well as for ensuring an $(\varepsilon,\delta)$-PAC bound of the sample Monte Carlo estimator of a given target integral $\pi(f)$, again for polynomially bounded functions $f$.

\section{Subexponential ergodicity of the diffusion}\label{sec:basics}
We now give an exact quantification of the stability of $\X$, which underlies the arguments from \cite{pardoux01} and also plays the central technical role in our approach. For details on terms from Markov stability theory such as petite sets or Harris recurrence, we refer to \cite{douc2009}.  

Define $q_+ = q \vee 0$. In the following, choose  $\iota = \iota(q_+) > 0$ small enough such that $\mathfrak{r} > \iota \lambda_+ (1-q_+)/2$. 
In this framework, it was shown in \cite[Proposition 5.1, Theorem 5.4]{douc2009} as a refinement of results in \cite{maly00} that $\X$ possesses a unique invariant distribution $\mu$ and that there exists some constant $C(q_+)$ such that, for $V_q(x) \coloneqq \exp(\iota \lVert x \rVert^{1-q})$, we have 
\begin{equation}\label{eq:subexp_tv}
\big\lVert \PP^x(X_t \in \cdot) - \mu \big\rVert_{\mathrm{TV}} \leq C(q_+) V_{q_+}(x) (1+t)^{\frac{2q_+}{1+q_+}} \mathrm{e}^{-(\iota^\prime t)^{(1-q_+)/(1+q_+)}}, \quad x \in \R^d, t \geq 0,
\end{equation} 
with $\iota^\prime \coloneqq \iota^{(1+q_+)/(1-q_+)}(1+q_+)(\mathfrak{r}- \lambda_+ \iota(1-q_+)/2)$ and $\lVert \nu \rVert_{\mathrm{TV}} \coloneqq \sup_{\lVert f \rVert_\infty \leq 1} \lvert \nu(f) \rvert$ for a signed finite measure $\nu$. 
Thus, for $q \in (0,1)$, $\X$ is subexponentially ergodic, and in case $q \in [-1,0]$ exponentially ergodic, i.e., 
\[\big\lVert \PP^x(X_t \in \cdot) - \mu \big\rVert_{\mathrm{TV}} \lesssim \mathrm{e}^{\iota \lVert x \rVert} \mathrm{e}^{-\iota(\mathfrak{r}-\lambda_+\iota/2) t}, \quad x \in \R^d, t \geq 0.\]
We remark that in \cite{douc2009} only the subexponentially ergodic case $q \in (0,1)$ is explicitly treated, but the arguments extend straightforwardly to the exponentially ergodic regime $q =0$, which then extend to $q \in (0,-1]$. 
Moreover, \cite[Theorem 5.3]{douc2009} establishes that $V_{q+} \in L^1(\mu)$, i.e., 
\begin{equation}\label{eq:subexp_mom2}
\E^\mu[V_{q_+}(X_0)] = \int_{\R^d} \exp(\iota \lVert x \rVert^{1-q_+}) \, \mu(\diff{x}) < \infty.
\end{equation}

It will be central for us to trade off subexponential ergodicity at a slower temporal rate with a less punishing penalty function. 
Let $\tilde{V}_\alpha(x) \coloneqq 1+ \lVert x \rVert^\alpha$ for $\alpha \geq 0$. 
Then, for $\zeta > 0$ and $\gamma > 2(1+ \zeta)$, Proposition 1 in \cite{pardoux01} demonstrates that 
\[
\big\lVert \PP^x(X_t \in \cdot) - \mu \big\rVert_{\mathrm{TV}} \leq C(\gamma,\zeta) \tilde{V}_{\gamma}(x)  (1+t)^{-(1+\zeta)}, \quad x \in \R^d, t \geq 0,
\] 
i.e., polynomial convergence with polynomial penalty function whose degree depends on the degree of the temporal rate that can be freely chosen.
We will  need to make use of polynomial convergence with respect to a stronger norm than the total variation norm considered above. 
To this end, let $H_1, H_2$ be a pair of Young functions on $\R_+$, which are in particular invertible and satisfy 
\begin{equation}\label{eq:young}
xy \leq H_1(x) + H_2(y), \quad x,y \geq 0,
\end{equation}
and let $\mathcal{I}$ be the family of pairs of inverse Young functions augmented by  $(\one,\mathrm{Id})$ and $(\mathrm{Id},\one)$. 
The prototypical example for such pairs are $H_1(x)= x^p/p, H_2(x) = x^q/q$ with $p,q$ conjugate Hölder exponents such that $1/p + 1/q = 1$, in which case \eqref{eq:young} is simply Young's inequality. 
More generally, one may pair any convex function with its Legendre transform to obtain \eqref{eq:young}. 

Following earlier work on discrete and continuous-time Markov models \cite{douc04,douc08,jarner02,tuominen94,fort05}, such inverse Young functions are used in \cite{douc2009} for subgeometrically ergodic Markov models to quantify the trade-off between speed of convergence and strength of the underlying \emph{$f$-norm}, which we introduce next. 
For a measurable function $f \geq 1$ and a signed measure $\nu$ on $(\R^d,\mathcal{B}(\R^d))$, its $f$-norm is defined by 
\[\lVert \nu \rVert_f \coloneqq \sup_{\lvert g \rvert \leq f} \lvert \nu(g) \rvert.\]
In particular, $\lVert \cdot \rVert_{\mathrm{TV}} = \lVert \cdot \rVert_{\bm{1}}$. 
Let us also define the $\delta$-delayed first hitting time of a set $B \in \mathcal{B}(\R^d)$ by $\tau_B(\delta) = \inf\{t \geq \delta: X_t \in B\}$, for $\delta \geq 0$, with $\tau_B = \tau_B(0)$. 
Moreover, we say that $B$ with $\mu(B) > 0$ is accessible, since $\mu$ is a maximal irreducibility measure of $\X$. 
Also note that $\mu$ as the invariant distribution of a Feller process is maximal Harris, such that in particular for any accessible set $B$ we have $\PP^x(\tau_B(\delta) < \infty) = 1$ for any $x \in \R^d$, $\delta \geq 0$. 

With the techniques from \cite{douc2009}, we obtain the following result on polynomial $f$-norm convergence and modulated moments, whose proof is given in Appendix \ref{app:conv}.
This explicit ergodicity result is central both for our derivation of the concentration inequality for continuous data and for its subsequent discrete extension.
It will turn out that appropriate choices for the pairing of Young functions to optimize the trade-off between convergence rate and strength of the $f$-norm will be essential when dealing with polynomially bounded test functions. 
We therefore truly need the full generality of the statement, which underlines the power of the approach in \cite{douc2009} for concrete applications.

\begin{proposition} \label{prop:f_conv}
Let $\gamma \geq 1+ q$ and $q \in (-1,1)$. 
Then, there exist functions $r_{\gamma,q}(t) \sim (1+t)^{(\gamma - (1+q))/(1+q)}$ and $f_{\gamma,q} \sim \tilde{V}_{\gamma - (1+q)}$ such that, for any pair of inverse Young functions $\Psi =(\Psi_1,\Psi_2) \in \mathcal{I}$ and some constant $C(\Psi)$, we have 
\begin{equation}\label{eq:poly_f}
(\Psi_1(r_{\gamma,q}(t)) \vee 1)\lVert \PP^x(X_t \in \cdot) - \mu \rVert_{\bm{1} \vee \Psi_2 \circ f_{\gamma,q}} \leq C(\Psi) \tilde{V}_\gamma(x), \quad x \in \R^d, t\geq 0,
\end{equation}
and, for any accessible set $B \in \mathcal{B}(\R^d)$ and any $\delta > 0$, there exists a constant $c(\Psi) > 0$ such that 
\begin{equation}\label{eq:modmoment}
\E^x\Big[\int_0^{\tau_B(\delta)} \Psi_1(r_{\gamma,q}(t)) \Psi_2(f_{\gamma,q}(X_t)) \diff{t} \Big] \leq c(\Psi)\tilde{V}_\gamma(x), \quad x \in \R^d.
\end{equation}
Moreover, if $q=-1$ and $\gamma > 0$, then, for any $\alpha \in (0,\mathfrak{r}\gamma)$, there exist a function $r_{\alpha}(t) \sim \exp(-\alpha t)$ and $f_\gamma \sim \tilde{V}_\gamma$ such that \eqref{prop:f_conv} and \eqref{eq:poly_f} are true with $f_{\gamma,q}$ and $r_{\gamma,q}$ replaced by $f_\gamma$ and $r_\alpha$, respectively.
\end{proposition}
Let us note that, for any $\eta \geq 0$, 
\[
\sup_{t \geq 0} \E^x\big[\lVert \tilde{X}_t \rVert^\eta\big] \leq c(\eta) \big(1+ \lVert x \rVert^\eta\big), \quad x \in \R^d,
\] 
where $\tilde{X}_t = X_{\tau^{-1}(t)}$ for the time change $\tau(t) \coloneqq \int_0^t \lVert \sigma^\top(X_s) X_s\slash \lVert X_s \rVert \rVert^2 \diff{s}$, cf.\ \cite[Proposition 1]{pardoux01}. Setting $\Psi_1 = \bm{1}$ and $\Psi_2 = \mathrm{Id}$ in \eqref{eq:poly_f}, it follows for the process on its unchanged time scale  that, for any $\eta > 0$,
\begin{equation}\label{eq:poly_mom2}
\sup_{t \geq 0} \E^x\big[\lVert X_t \rVert^\eta \big] \leq \mathfrak{C}(\eta) (1+ \lVert x \rVert^{\eta + 1 +q}), \quad x \in \R^d.
\end{equation}

\section{Concentration of additive diffusion functionals}\label{sec:conc}
Recall the definition of the scaled additive functionals $\mathbb{G}_t(f)$ and $\mathbb{G}_{n,\Delta}(f)$ from \eqref{eq:add_cont} and \eqref{eq:add_disc}, respectively. Motivated by the existence of a regular solution to the Poisson equation for polynomially bounded functions, we study deviations of $\mathbb{G}_t(f)$ and its discrete version $\mathbb{G}_{n,\Delta}(f)$ for functions $f$ belonging to the function class $\mathcal{F}(\eta, \mathfrak{L})$ given by
\[\mathcal{F}(\eta, \mathfrak{L}) \coloneqq \big\{\tilde{f} - \mu(\tilde{f}): \tilde{f} \in \mathcal{G}(\eta, \mathfrak{L})\big\},
\]
for 
\[\mathcal{G}(\eta, \mathfrak{L}) \coloneqq \big\{f \colon \R^d \to \R: \lvert f(x) \rvert \leq \mathfrak{L}(1+\lVert x \rVert^\eta), \, x \in \R^d\},\]
for some finite constant $\mathfrak{L} > 0, \eta \geq 0$.

There is a vast amount of literature on concentration inequalities for path integrals of general Markov processes. The most powerful results are generally established under the assumption of functional inequalities such as Poincaré or log-Sobolev. However, the elliptic diffusions considered in this paper generally do not satisfy such rather strong functional inequalities. In this regard, \cite{cattiaux05} establish concentration inequalities for bounded functionals under a so-called weak Poincaré inequality, which is demonstrated to be equivalent to an $\alpha$-mixing assumption on the process, cf.\ \cite[Proposition 3.4]{cattiaux05}. 
Recall that a stationary Markov process $(Y_t)_{t \geq 0}$ with natural filtration $(\mathcal{F}_t)_{t \geq 0}$ and initial distribution $\nu$ is said to be $\alpha$-mixing if the mixing coefficient $\alpha_\nu(t) \coloneqq \sup_{s \geq 0} \sup_{A \in \mathcal{F}_s, B \in \mathcal{F}_{s+t}} \lvert \PP^\nu(A \cap B) - \PP^\nu(A)\PP^\nu(B) \rvert$ tends to zero as $t \to \infty$. It follows from \eqref{eq:subexp_tv} and \eqref{eq:subexp_mom2} that the stationary $\beta$-mixing coefficient $\beta(t) \coloneqq \int_{\R^d} \lVert \PP^x(X_t \in \cdot)- \mu\rVert_{\mathrm{TV}} \, \mu(\diff{x})$ of our diffusion process satisfies 
\[\beta(t) \leq c \exp\big(-\iota^{\prime\prime} t^{(1-q_+)/(1+q_+)}\big),\] 
for any $\iota^{\prime \prime} \in (0,\iota^\prime)$ and some constant $c$ depending on $\iota^{\prime\prime}$, i.e., the stationary diffusion $\X$ is subexponentially $\beta$-mixing. Consequently, using the well-known fact that $\alpha_\mu(t) \leq \beta(t)$,  \cite[Proposition 3.9]{cattiaux05} yields the following result for \textit{bounded} $f$.
\begin{theorem}{\cite[Proposition 3.9]{cattiaux05}} \label{theo:cattiaux}
For $\iota^{\prime\prime} \in (0,\iota^\prime)$, define  
\begin{equation}\label{def:c}
c(q,\iota^{\prime\prime}) \coloneqq ((1+q_+)/(1-q_+))^{1/(1-q_+)}((1-q_+)\iota^{\prime \prime}/(1+q_+))^{(1+q_+)/(2(1-q_+))}/2.
\end{equation}
For any such $\iota^{\prime\prime}$, there exists a constant $\mathfrak{c} > 0$ such that, for all $f \in \mathcal{F}(0,\mathfrak{L})$ and $(u,t) \in \R^2_+$ such that
\begin{equation}\label{cond:u}
\mathfrak{c} (1+q_+)(1-q_+)^{-(1-q_+)/2} \leq u < \big(c(q,\iota^{\prime\prime})\lfloor t \rfloor/ \sqrt{t}\big)^{1-q_+},
\end{equation} 
it holds
\[ \PP^\mu\Big(\lvert \mathbb{G}_t(f) \rvert > 2\mathfrak{L}\big(c(q,\iota^{\prime\prime})^{-1} u^{\frac{1}{1-q_+}} + t^{-1/2}\big)\Big) \leq 2\mathrm{e}^{-u}.\]
\end{theorem}
In the above result, the restriction on $u$ in \eqref{cond:u} is explained by the proof technique that makes use of general moment bounds for discrete $\alpha$-mixing sequences from \cite{rio17}. This approach requires the integral to be divided into a finite number of blocks with a carefully chosen length that determines the degree of mixing of the block sequence. 

In the following, we add to this result  by allowing polynomially growing integrands $f$. It is well-known that dropping the boundedness assumption poses  major challenges in deriving concentration inequalities, some of which have been elegantly solved in \cite{gao14} for symmetric Markov processes satisfying  (strong) functional inequalities. It should also be noted that in \cite[Section 3.2]{cattiaux05} some arguments are provided how conclusions for unbounded integrands $f$ can be drawn from Theorem \ref{theo:cattiaux} by employing a truncation technique. However, there appears to be a gap in the proposed strategy, which prevents it from being applicable for $u > 0$ such that $u/\sqrt{t}$ is small. Since our ultimate focus is on applications of our concentration inequalities to the inference of PAC bounds for $t^{-1/2}\mathbb{G}_t(f)$, we do not further pursue an approach relying on discrete mixing results, but employ a different technique that is embedded more naturally in the continuous framework.

\subsection{Continuous observations}\label{sec:conti}
Our main result for continuous observations is the following exponential concentration bound for polynomially bounded functions.

\begin{theorem} \label{prop:conc}
There exists a constant $\mathfrak{W}$, depending on $q, \eta$ and the diffusion coefficients $b$ and $\sigma$, such that, for any $p \geq 2$, $t > 0$ and $f  \in \mathcal{F}(\eta, \mathfrak{L})$, we have
\begin{equation}\label{eq:mom}
\lVert \mathbb{G}_t(f) \rVert_{L^p(\PP^\mu)} \leq \mathfrak{L}\mathfrak{W} p^{\frac{1}{2} + \frac{\eta + q^\prime + q +1}{1-q_+}}.
\end{equation}
As a consequence, for any $t > 0$,
\begin{equation}\label{eq:conc}
\PP^\mu\Big(\lvert \mathbb{G}_t(f) \rvert > \mathrm{e}\mathfrak{L}\mathfrak{W} u^{\frac{1}{2} + \frac{\eta + q^\prime + q +1}{1-q_+}} \Big) \leq \mathrm{e}^{-u}, \quad u \geq 2.
\end{equation}
\end{theorem}

The proof will be given by combining a sequence of technical lemmas that we develop in the following. An interpretation of the result will be stated later in Remark \ref{rem:conc} since this requires making explicit reference to the proof. The first result that we need are bounds on the $L^p$-norms of the invariant measure $\mu$ which are implied by its subexponential tails.

\begin{lemma}\label{lem:l_p}
For all $p \geq 1$, it holds that 
\[\E^\mu\big[\lVert X_0 \rVert^p\big]^{1/p} \leq c_{q_+} p^{1/(1-q_+)},\]
where 
\[c_{q_+} = \e^{\e/2 + (1-q_+)/12} ((1-q_+)\iota\e)^{-1/(1-q_+)}\sqrt{\frac{2\pi}{1-q_+}} \E^\mu[V_{q+}(X_0)].\]
\end{lemma}
\begin{proof}
Let $a_{q_+} = ((1-q_+)\iota\e)^{-1/(1-q_+)}$. 
Using Markov's inequality, it follows that, for any $u \geq 1$, 
\[\PP^\mu\big(\lVert X_0 \rVert \geq \e^{1/(1-q_+)}a_{q_+} u \big) \leq \E^\mu[V_{q_+}(X_0)] \exp\big(-u^{1-q_+}/(1-q_+)\big),\]
with $\E^\mu[V_{q_+}(X_0)] < \infty$ due to \eqref{eq:subexp_mom2}. The assertion now follows from \cite[Proposition 7.13]{foucart13}.
\end{proof}

Next, we state the martingale approximation of the additive functional $\mathbb{G}_t(f)$ for polynomially bounded $f$ with the help of the It\={o}--Krylov formula, which extends the usual It\={o} formula for diffusion processes from functions with $\mathcal{C}^2$-regularity to functions with slightly weaker Sobolev regularity. 
This is necessary in light of the regularity of the solution to the Poisson equation that we described in Section \ref{sec:basics}.

\begin{lemma}\label{lem:mart_approx}
For any $f \in \mathcal{F}(\eta, \mathfrak{L})$, we have a decomposition 
\[\mathbb{G}_t(f) = \frac{1}{\sqrt{t}} \mathbb{M}_t(f) + \frac{1}{\sqrt{t}} \mathbb{R}_t(f),\]
where $(\mathbb{M}_t(f))_{t \geq 0}$ is a continuous square-integrable $\PP^\mu$-martingale and both $f \mapsto \mathbb{M}_{\Cdot}(f)$ and $f \mapsto \mathbb{R}_{\Cdot}(f)$ are linear. 
Moreover, there exists a global constant $c \geq 1$
such that, for any $p\geq 1, t \geq 0$
\begin{equation}\label{eq:martingale_bound}
\E^\mu\big[\lvert \mathbb M_t(f) \rvert^p\big]^{1/p} \leq c \lambda_+^{1/2} p^{1/2} \sqrt{t}  \big\lVert \lVert \nabla L^{-1}[f] \lVert \big\rVert_{L^{p \vee 2}(\mu)},
\end{equation}
and 
\begin{equation}\label{eq:remainder_bound}
\E^\mu\big[\lvert \mathbb{R}_t(f) \rvert^p \big]^{1/p} \leq 2 \lVert L^{-1}[f] \rVert_{L^p(\mu)}.
\end{equation}
\end{lemma}
\begin{proof}
For any $p\ge1$, $L^{-1}[f] \in \mathcal{W}^{2,p}_{\mathrm{loc}}(\R^d)$, the coefficients $b,\sigma$ are locally bounded, $\sigma \sigma^\top$ is uniformly positive definite and \eqref{eq:subexp_mom2} guarantees $\E^\mu[\int_0^t \lVert b(X_s) \rVert^2 \diff{s}] < \infty$ since $\lVert b \rVert \lesssim \tilde{V}_{q^\prime}$.
Thus, we can apply the It\={o}--Krylov formula (cf.\ \cite[Theorem 2.10.1]{krylov09}) 
to obtain the $\PP^{\mu}$-a.s.\ identities 
\begin{align*}
L^{-1}[f](X_t) &= L^{-1}[f](X_0) + \int_0^t \big(\nabla L^{-1}[f](X_s)\big)^\top \sigma(X_s)\diff{W_s} + \int_0^t LL^{-1}[f](X_s) \diff{s}\\
&= L^{-1}[f](X_0) + \int_0^t \big(\nabla L^{-1}[f](X_s)\big)^\top \sigma(X_s) \diff{W_s} + \int_0^t f(X_s) \diff{s}, \quad t \geq 0.
\end{align*}
Here, the second equality follows from $LL^{-1}[f] = f$, $\lebesgue$-a.e., and $\PP^x(X_t \in \cdot) \ll \lebesgue$ for any $(t,x) \in (0,\infty) \times \R^d$, which implies
\[\E^x\Big[\Big\vert \int_0^t LL^{-1}[f](X_s) \diff{s} - \int_0^t f(X_s) \diff{s} \Big\vert \Big] \leq \int_0^t \int_{\R^d} \lvert LL^{-1}[f](y) - f(y) \rvert p_s(x,y) \diff{y} \diff{s} = 0, \quad x\in \R^d,\]
such that $\int_0^t LL^{-1}[f](X_s) \diff{t} = \int_0^t f(X_s) \diff{s}$ $\PP^x$-a.s.\ for any $x \in \R^d$ and hence also $\PP^\mu$-a.s.\ follows. 
Consequently,
\[\mathbb G_t(f) = \frac{1}{\sqrt{t}}\int_0^t f(X_s) \diff{s} = \frac{1}{\sqrt{t}} \mathbb M_t(f) + \frac{1}{\sqrt{t}} \mathbb R_t(f), \quad t \geq 0,\]
where 
\[\mathbb M_t(f) = -\int_0^t \big(\nabla L^{-1}[f](X_s)\big)^\top \sigma(X_s)\diff{W_s}, \quad t \geq 0,\]
and 
\[\mathbb{R}_t(f) = L^{-1}[f](X_t) - L^{-1}[f](X_0), \quad t \geq 0.\]
The square-integrable martingale property of $\mathbb{M}_\cdot(f)$ follows from
\[\lVert \nabla L^{-1}[f](x) \rVert \lesssim \tilde{V}_{\eta + q^\prime + q + 1}(x), \quad x \in \R^d,\]
which is demonstrated in the proof of Lemma \ref{lem:bound_poisson}, such that \eqref{eq:subexp_mom2} implies that 
\[\int_0^t \E^\mu\big[\big\lVert \big(\nabla L^{-1}[f](X_s)\big)^\top \sigma(X_s)\big\rVert^2\big] \diff{s} \lesssim t \lambda_+ \mu\big(\tilde{V}_{2(\eta+q^\prime + q + 1)}\big) < \infty, \quad t \geq 0.\]
The bound \eqref{eq:remainder_bound} is now an immediate consequence of stationarity under $\PP^\mu$ and Minkowski's inequality.
Moreover, using the Burkholder--Davis--Gundy inequality for continuous martingales started in $0$ in the form given in \cite[Proposition 4.2]{barlow82}, it follows that, for some $c \geq 1$ and any $p \geq 1$,
\begin{equation}\label{eq:mart1}
\begin{split}
\E^\mu\big[\lvert \mathbb M_t(f) \rvert^p\big] &\leq c^pp^{p/2} \E^\mu \big[\langle \mathbb{M}_{\Cdot}(f) \rangle_t^{p/2}\big]\\
&= c^pp^{p/2} \E^\mu\Big[\Big(\int_0^t \lVert \sigma^\top(X_s) \nabla L^{-1}[f](X_s) \rVert^2 \diff{s}\Big)^{p/2} \Big]\\
&\leq c^pp^{p/2} \lambda_+^{p/2} \E^\mu\Big[\Big(\int_0^t \lVert  \nabla L^{-1}[f](X_s) \rVert^2 \diff{s}\Big)^{p/2} \Big].
\end{split}
\end{equation}
Consequently, for $p \geq 2$, first using Jensen's inequality and then Fubini together with stationarity gives 
\begin{align*}
\E^\mu\big[\lvert \mathbb M_t(f) \rvert^p\big] &\leq c^p p^{p/2}\lambda_+^{p/2} t^{p/2} \E^\mu\Big[\frac{1}{t}\int_0^t \lVert  \nabla L^{-1}[f](X_s) \rVert^p \diff{s} \Big]\\
&= c^p p^{p/2}\lambda_+^{p/2} t^{p/2} \big\lVert \nabla L^{-1}[f]\big\rVert^p_{L^p(\mu)} ,
\end{align*}
and hence 
\[\E^\mu\big[\lvert \mathbb M_t(f) \rvert^p\big]^{1/p} \leq c \lambda_+^{1/2} p^{1/2} \sqrt{t} \big\lVert \nabla L^{-1}[f]\big\rVert_{L^p(\mu)}, \quad t\geq 0.\]
In case $p \in [1,2)$, we get from \eqref{eq:mart1} with another application of Jensen's inequality and Fubini
\begin{align*} 
\E^\mu\big[\lvert \mathbb M_t(f) \rvert^p\big] &\leq c^pp^{p/2} \lambda_+^{p/2}\E^\mu \Big[\frac{1}{t}\int_0^t \big\lVert \nabla L^{-1}[f](X_s) \big\rVert^{2} \diff{s} \Big]^{p/2}\\
&= c^pp^{p/2} \lambda_+^{p/2} t^{p/2} \big\lVert \lVert\nabla L^{-1}[f] \lVert \big\rVert^{p}_{L^{2}(\mu)}.
\end{align*}
Thus, for any $p \geq 1$, \eqref{eq:martingale_bound} follows.
\end{proof}

In view of \eqref{eq:martingale_bound} and \eqref{eq:remainder_bound}, to exploit the martingale approximation we need concrete bounds on the solution of the Poisson equation $L^{-1}[f]$ and its gradient $\nabla L^{-1}[f]$. 
This is the content of the next lemma, which can essentially be obtained from combining Lemma \ref{lem:l_p} with the Sobolev estimates  from \cite{pardoux01} and \cite{bogachev18}. 
For later reference and some clarification concerning the role of the drift growth, we give a full proof that simplifies some arguments from \cite{pardoux01} thanks to Proposition \ref{prop:f_conv}. 

\begin{lemma} \label{lem:bound_poisson}
Let $p \geq 1$. 
There exist constants $\mathfrak{U}(q,\eta), \mathfrak{V}(q,q^\prime,\eta)$ (independent of $p$) such that, for any $f = \tilde{f} - \mu(\tilde{f}) \in \mathcal{F}(\eta,\mathfrak{L})$,
\begin{equation}\label{eq:bound_poisson}
\big \lVert L^{-1}[f] \big\rVert_{L^{p}(\mu)} \leq \mathfrak{L}\mathfrak{U}(q,\eta) p^{\frac{\eta + q +1}{1-q_+}} 
\end{equation}
and
\begin{equation}\label{eq:bound_gradient}
\big \lVert \lVert \nabla L^{-1}[f] \lVert \big\rVert_{L^{p}(\mu)} \leq \mathfrak{L}\mathfrak{V}(q,q^\prime,\eta)   p^{\frac{\eta + q^\prime + q +1}{1-q_+}}.
\end{equation}
\end{lemma}
\begin{proof}
By a slight adjustment to the proof of Proposition 4.1 in \cite{bogachev18}\footnote{as the authors point out, the gradient bounds derived in \cite[Theorem 1]{pardoux01} are only valid in case of bounded drift}, we obtain for any $r > d$
\begin{equation}\label{eq:sobolev0}
\big\lVert \nabla L^{-1}[f](x) \big\rVert \lesssim \big(1+ \sup_{y \in B(x,1)} \lvert b(y) \rvert \big)\lVert L^{-1}[f] \rVert_{L^{r}(B(x,1))} + \lVert f \rVert_{L^{r}(B(x,1))}.
\end{equation}
Therefore, using Hölder's inequality and the growth condition on the drift $b$,
\begin{align} \label{eq:sobolev}
\big \lVert \lVert \nabla L^{-1}[f] \lVert \big\rVert_{L^{p}(\mu)} \lesssim \lVert (1+ \lVert \cdot \rVert^{q^\prime}) \rVert_{L^{2p}(\mu)}\big\lVert \lVert L^{-1}[f] \rVert_{L^r(B(\cdot,1))} \big \rVert_{L^{2p}(\mu)} + \big\lVert \lVert f \rVert_{L^r(B(\cdot,1))} \big \rVert_{L^p(\mu)}.
\end{align}
If $q > -1$, let $\gamma > 2(1+ q)$. Then we can calculate as in the proof of \cite[Theorem 1]{pardoux01} to obtain
\begin{align*}
\lvert L^{-1}[f](x) \rvert &\leq  \int_0^\infty \int_{\R^d} \lvert \tilde{f}(y)\rvert \lvert p_t(x,y) - \rho(y) \rvert \diff{y} \diff{t}\\
&\leq  \int_0^\infty \Big(\int_{\R^d} \lvert p_t(x,y) - \rho(y) \rvert \diff{y}\Big)^{1/2} \Big(\int_{\R^d} \lvert \tilde{f}(y) \rvert^{2} (p_t(x,y) + \rho(y)) \diff{y}\Big)^{1/2} \diff{t} \Big)\\
&=  \int_0^\infty \big(\lVert P_t(x,\cdot) - \mu \rVert_{\mathrm{TV}}\big)^{1/2} \Big(\int_{\R^d} \lvert \tilde{f}(y) \rvert^{2} (p_t(x,y) + \rho(y)) \diff{y}\Big)^{1/{2}} \diff{t} \Big)\\
&\leq  \mathfrak{L}C^\prime(\gamma,q) (\tilde{V}_\gamma(x))^{1/2}  \int_0^\infty (1+t)^{-\frac{\gamma - (1+q)}{1+q}} \Big(\int_{\R^d} (1+ \lVert y \rVert^\eta)^{2} (p_t(x,y) + \rho(y)) \diff{y}) \Big)^{1/2} \diff{t}\\
&\leq  \mathfrak{L}C(\eta,\gamma,q)(\tilde{V}_\gamma(x))^{1/2} (1 + \lVert x \rVert^{\eta + (1+q)/2}) \int_1^\infty t^{-(\gamma - (1+q))/(1+q)} \diff{t}\\
&=  \mathfrak{L}C^\prime(\eta,\gamma,q)  \tilde{V}_{\eta + (1+q +\gamma)/2}(x),
\end{align*}
where we used Cauchy--Schwarz for the second inequality and \eqref{eq:poly_f} for the third inequality. The last inequality arises from \eqref{eq:subexp_mom2} and \eqref{eq:poly_mom2}. 
A similar calculation, using exponential ergodicity with arbitrary polynomial penalty $\tilde{V}_\gamma$ in case $q = -1$ with any $\gamma > 0$, shows that the above estimate remains valid for $q = -1$.
By the strong Markov property, we have for any $R > 0$ and $\tau_R \coloneqq \tau_{\overbar{B(0,R)}}$,
\[L^{-1}[f](x) = \E^x\big[L^{-1}[f](X_{\tau_R})\big] + \E^x\Big[\int_0^{\tau_R} f(X_t) \diff{t}\Big].\]
By the above, $u$ is locally bounded and thus the first term is bounded for any $R > 0$. 
For the second term, we can employ the It\={o} formula argument from \cite[Theorem 2]{pardoux01} to improve this bound to $\lvert L^{-1}[f]\rvert \lesssim \mathfrak{L}\tilde{V}_{\eta + 1 +q}$. 
Alternatively, apart from the case $\eta = 0$, $q =-1$, we may simply note that, by setting $\Psi_1 = \bm{1}$ and $\Psi_2 = \mathrm{Id}$, \eqref{eq:modmoment} yields that for any $\delta > 0$
\[\Big\vert \E^x\Big[\int_0^{\tau_R} f(X_t) \diff{t}\Big] \Big\vert \leq  \E^x\Big[\int_0^{\tau_R(\delta)} \lvert f(X_t) \rvert \diff{t}\Big] \leq \mathfrak{L}C(\eta,q,R,\delta) \tilde{V}_{\eta + 1 +q}(x), \quad x \in \R^d.\]
Here we used that $\lvert f \rvert \lesssim \mathfrak{L}\tilde{V}_{\eta}$ and that $B(0,R)$ is accessible, which follows from $\lebesgue$-irreducibility of $\X$ implied by uniform positive definiteness of $\sigma \sigma^\top$, cf.\ \cite[Theorem 2.3]{stramer97}. 
Thus, 
\[\lvert L^{-1}[f](x) \rvert \leq C\mathfrak{L}\tilde{V}_{\eta + 1+q}(x), \quad x \in \R^d,\]
follows for some constant $C$ depending on $\eta$ and $q$.
Consequently, for any $p\geq 1$, Lemma \ref{lem:l_p} yields
\[\lVert L^{-1}[f] \rVert_{L^p(\mu)} \leq  C\mathfrak{L} \lVert \tilde{V}_{\eta + 1 +q} \rVert_{L^p(\mu)} \leq C\mathfrak{L} c_q^{\eta + q + 1} p^{\frac{\eta + q +1}{1-q_+}}.\]
Using that 
\[\lvert \tilde{V}_{\eta + q +1 }(x+y) \rvert \leq 2^{\eta + q} (\tilde{V}_{\eta + q+1}(x) + \tilde{V}_{\eta + q+1}(y)) \leq 2^{\eta + q} (2+ \tilde{V}_{\eta + q+1}(x)), \quad x \in \R^d, y \in B(0,1),\]
it also follows that 
\[\big\lVert \lVert L^{-1}[f] \rVert_{L^r(B(\cdot,1))} \big \rVert_{L^{2p}(\mu)} \leq C\mathfrak{L} 2^{\eta+q+1} \lVert \tilde{V}_{\eta + q +1} \rVert_{L^{2p}(\mu)} \leq   \mathfrak{L}C 2^{\eta+q+1} c_q^{\eta + q +1} (2p)^{\frac{\eta + q+1}{1-q_+}}.\]
Moreover, $\lvert \tilde{f}(x+y) \rvert \leq \mathfrak{L}2^\eta(2+ \tilde{V}(x))$ for $y \in B(0,1)$ implies
\begin{align*}
\lVert \lVert f \rVert_{L^{r} B(\cdot,1)} \rVert_{L^{p}(\mu)}  &\leq \mathfrak{L} 2^{\eta} \big(2 + (1 + \lebesgue(B(0,1)))\lVert \tilde{V}_{\eta} \rVert_{L^p(\mu)} \big)\\
&\leq  \mathfrak{L}2^\eta (1 \vee \lebesgue(B(0,1))) (1+ c_{q+}^\eta p^{\eta/(1-q_+)}),
\end{align*}
such that \eqref{eq:sobolev} allows us to conclude that 
\[\big \lVert \lVert \nabla L^{-1}[f] \lVert \big\rVert_{L^{p}(\mu)} \leq \mathfrak{L}\mathfrak{V}(q,q^\prime,\eta) p^{\frac{\eta+q^\prime + q +1}{1-q_+}}.\]
\end{proof}

We are now ready to infer Theorem \ref{prop:conc} from the previous results.
\begin{proof}[Proof of Theorem \ref{prop:conc}]
The moment bounds \eqref{eq:mom} are an immediate consequence of the combined statements of Lemma \ref{lem:mart_approx} and Lemma \ref{lem:bound_poisson}. By Markov's inequality, \eqref{eq:mom} implies \eqref{eq:conc}.
\end{proof}
\begin{remark}\label{rem:conc}
It would be desirable that the concentration rate provided by Theorem \ref{prop:conc} matches the rate in Theorem \ref{theo:cattiaux} for the bounded case $\eta = 0$ and the rates for polynomially growing integrands for \textit{scalar}, exponentially ergodic diffusions with at most linear drift (i.e., $d=1, q=0, q^\prime = 1$) from \cite[Proposition 7]{aeckerle21}. 
The reason for the gap in the rate can be traced down to the Sobolev estimates \eqref{eq:sobolev0}, where the gradient $\nabla L^{-1}[f]$ is bounded in terms of $L^{-1}[f]$. In contrast, the strategy in \cite{aeckerle21}, see also \cite{galt07}, works in the other direction. That is, by exploiting the explicit solution of the Poisson equation in $d=1$, tight pointwise bounds on the gradient $\nabla L^{-1}[f]$ are established first, which are then used to bound the remainder term $L^{-1}[f](X_t) - L^{-1}[f](X_0) = \int_{X_0}^{X_t} \nabla{L}^{-1}[f](x) \diff{x}$ in the martingale approximation. Such a strategy is not feasible in the multivariate setting since $L^{-1}[f]$ is not explicitly known. Improving our concentration result Theorem \ref{prop:conc} would therefore require tighter estimates on the solution of the Poisson equation and its gradient than those that can be achieved with the ideas from \cite{pardoux01}. This is a challenging and interesting question for future research.
\end{remark}

\paragraph{Stationary and non-stationary PAC bounds}
As an immediate consequence of Theorem \ref{theo:cattiaux} and Theorem \ref{prop:conc}, we can derive the following quantitative version of the ergodic theorem and a stationary PAC bound for (sub-) geometric diffusions. Let us define the rate function
\[\varsigma(\eta,q,q^\prime) \coloneqq \begin{cases} 1-q_+, &\text{ if } \eta = 0,\\ \frac{1}{2} + \frac{\eta + q^\prime + q +1}{1-q_+}, &\text{ if } \eta > 0, \end{cases}\] 
and the sample length function
\[\Psi(\varepsilon,\delta) \coloneqq \begin{cases} \Big(\frac{c(q,\iota^{\prime\prime})^{-1}(\log(2/\delta))^{1/(1-q_+)}+1}{1 \wedge \varepsilon/(2\mathfrak{L})} \Big)^2, &\text{ if } \eta = 0,\\\bigg(\frac{\mathrm{e}\mathfrak{L}\mathfrak{W} (\log (1/\delta))^{\frac{1}{2}+\frac{\eta + q^\prime + q +1}{1-q_+}}}{\varepsilon}\bigg)^2, &\text{ if }\eta > 0,\end{cases}\]
with $c(q,\iota^{\prime\prime})$ defined in \eqref{def:c} and $\mathfrak W$ denoting the constant established in Theorem \ref{prop:conc}.

\begin{corollary}\label{coro:pac}
Let $f \in \mathcal{G}(\eta, \mathfrak{L})$, $\varepsilon > 0$ and $\delta \in (0,1)$ such that $\delta < 2\exp(-\mathfrak{c} (1+q_+)(1-q_+)^{-(1-q_+)/2})$ if $\eta = 0$ and $\delta < \mathrm{e}^{-2}$ if $\eta > 0$. 
Then, for $t \geq \Psi(\varepsilon,\delta)$, it holds that
\begin{equation}\label{eq:pac}
\PP^\mu\Big(\Big\vert \frac{1}{t}\int_0^t f(X_s) \diff{s} - \mu(f)  \Big\vert \leq \varepsilon \Big) \geq 1 - \delta.
\end{equation}
Moreover, for any increasing sequence $(t_n)_{n \in \N} \subset \R_+$ with $\inf_{n \in \N} (t_n - t_{n-1}) > 0$, it holds for any $\delta_0 > 0$ 
\begin{equation}\label{eq:almostsure1}
\lim_{n \to \infty} \sqrt{t_n}(\log t_n)^{-(\varsigma(\eta,q,q^\prime) + \delta_0)} \Big\lvert \frac{1}{t_n}\int_0^{t_n} f(X_s) \diff{s} - \mu(f)\Big\rvert = 0, \quad \PP^\mu\text{-a.s.}
\end{equation}
If $f$ is bounded, we even have, for any $\tilde{q} \in (q_+,1)$,
\[
\lim_{t \to \infty} \sqrt{t}(\log t)^{-\frac{1}{1-\tilde{q}}} \Big\lvert \frac{1}{t}\int_0^{t} f(X_s) \diff{s} - \mu(f)\Big\rvert = 0, \quad \PP^\mu\text{-a.s.}
\]
\end{corollary}
\begin{proof} 
The first two assertions immediately follow from Theorem \ref{theo:cattiaux} and Theorem \ref{prop:conc}. For any $\varepsilon > 0$ and $\delta>0$, there exists $t(\varepsilon) \geq \mathrm{e}$ such that, for any $t \geq t(\varepsilon)$, we have \[\mathfrak{L}\{(\mathrm{e}\mathfrak{W})\vee (c(q,\iota^{\prime\prime})^{-1}+1)\}(2 \log t)^{-\frac{\delta_0}{1-q_+}} \leq \varepsilon.
\]
Consequently, by Theorem \ref{theo:cattiaux} and Theorem \ref{prop:conc}, it follows that for 
\[U_t \coloneqq \sqrt{t}(2\log t)^{-(\varsigma(\eta,q,q^\prime) + \delta_0)} \Big( \frac{1}{t}\int_0^{t} f(X_s) \diff{s} - \mu(f)\Big)\]
and $t \geq t(\varepsilon)$ such that, in case $\eta = 0$, additionally $2 \log t \geq \mathfrak{c} (1+q_+)(1-q_+)^{-(1-q_+)/2}$,
\[\PP^\mu(\lvert U_t \rvert > \varepsilon) \leq \PP^\mu \Big(\lvert \mathbb{G}_t(f) \rvert > \mathfrak{L}\{(\mathrm{e}\mathfrak{W})\vee (c(q,\iota^{\prime\prime})^{-1}+1)\} (2\log t)^{\varsigma(\eta,q,q^\prime)}\Big) \leq t^{-2}.\]
Thus, 
\[\PP^\mu(\lvert U_t \rvert > \varepsilon) \leq \one_{[0,t(\varepsilon))} + t^{-2} \one_{[t(\varepsilon), \infty)} \eqqcolon g_{\varepsilon}(t), \quad t > 0.\]
Since $g_{\varepsilon} \in L^1(\R_+)$ and is decreasing, it follows for $a \coloneqq \inf_{n \in \N} (t_{n+1} - t_n) > 0$
\[\infty > \int_{t_n}^{\infty} g_\varepsilon(t) \diff{t} \geq \sum_{m \geq n} (t_{m+1} - t_m) g_\varepsilon(t_{m+1}) \geq a \sum_{m \geq n+1} g_\varepsilon(t_{m}) \geq a \sum_{m \geq n+1} \PP^\mu(\lvert U_{t_m} \rvert > \varepsilon).\]
Hence, for any $\varepsilon > 0$, $\sum_{n \in \N} \PP^\mu (\lvert U_{t_n} \rvert > \varepsilon) < \infty$ such that Borel--Cantelli implies $\lim_{n \to \infty} U_{t_n} = 0$, $\PP^\mu\text{-a.s.}$,
which gives \eqref{eq:almostsure1}. This argument is borrowed from the proof of Lemma 3.1 in \cite{bosq97}. By the same lemma, it follows from the above that we even have convergence along any sequence $(\tilde{t}_n)_{n \in \N}$, $\PP^\mu$-a.s., provided that the map $t \mapsto U_t$ is uniformly continuous $\PP^\mu$-a.s. This can be easily verified when $f$ is bounded (see, e.g., the proof of Proposition 4.3 in \cite{bosq97}), which proves the last assertion.
\end{proof}

To get a non-stationary PAC bound, we consider the burn-in sample average 
\[\mathbb{H}_{v,t}(f) \coloneqq \frac{1}{\sqrt{t}} \mathbb{G}_t(f) \circ \theta_v = \frac{1}{t}\int_{v}^{v+t} f(X_s) \diff{s}, \quad t > 0,v \geq 0,\]
with burn-in length $v$.
Our naming convention follows the MCMC literature, where a standard procedure of dealing with non-stationary simulation procedures is to run the simulation algorithm for a certain amount of time before collecting samples, which is usually referred to as the burn-in.

\begin{corollary} \label{coro:burnin}
Let $\varepsilon > 0, \delta \in (0,1)$ such that $\delta < 2\mathrm{e}^{-2}$ if $\eta > 0$ and $\delta < 4\exp(-\mathfrak{c} (1+q_+)(1-q_+)^{-(1-q_+)/2})$ if $\eta = 0$.
Let also $\nu$ be some probability distribution such that $V_{q_+} \in L^1(\nu)$. 
Choose some $\iota^{\prime\prime} \in (0,\iota^\prime)$, and define $C \coloneqq C(q_+)c(\iota^{\prime\prime})\lVert V_{q_+} \rVert_{L^{1}(\nu)}$, where $c(\iota^{\prime \prime})$ is some constant such that 
\[\forall t\geq 1:\quad (1+t)^{\frac{2q_+}{1+q_+}} \mathrm{e}^{-(\iota^\prime t)^{(1-q_+)/(1+q_+)}} \leq c(\iota^{\prime\prime})\mathrm{e}^{-(\iota^{\prime \prime} t)^{(1-q_+)/(1+q_+)}}.\]
Then, for $t \geq \Psi(\varepsilon,\delta/2)$ and burn-in length $v \geq 1\vee (\log (2C/\delta))^{(1+q_+)/(1-q_+)}/\iota^{\prime \prime}$, we have, for any $f \in \mathcal{G}(\eta, \mathfrak{L})$, 
\[\PP^\nu(\lvert \mathbb H_{v,t}(f) - \mu(f) \rvert \leq \varepsilon) \geq 1 - \delta.\]
\end{corollary}
\begin{proof} 
Under the given assumptions, \eqref{eq:subexp_tv} implies 
\begin{equation} \label{eq:subexp_tv3}
\lVert \PP^x(X_t \in \cdot) - \mu \rVert_{\mathrm{TV}} \leq C \frac{V_{q_+}(x)}{\lVert V_{q_+} \rVert_{L^1(\nu)}} \mathrm{e}^{-(\iota^{\prime\prime} t)^{(1-q_+)/(1+q_+)}}.
\end{equation}
Define $g(y) \coloneqq \PP^y(\lvert t^{-1}\int_0^t f(X_s) \diff{s} - \mu(f) \rvert > \varepsilon)$.
By the Markov property, \eqref{eq:subexp_tv3} and the magnitude of the burn-in $v$, for any $x \in \R^d$,
\begin{align*}
&\Big\vert \PP^x(\lvert \mathbb{H}_{v,t}(f) - \mu(f) \rvert > \varepsilon) - \PP^\mu\Big(\Big\vert \frac{1}{t}\int_0^t f(X_s) \diff{s} - \mu(f)  \Big\vert > \varepsilon \Big)\Big\vert \\
&\quad = \big\lvert \E^x[g(X_v)] - \mu(g) \big\rvert \leq \lVert \PP^x(X_v \in \cdot) - \mu \rVert_{\mathrm{TV}} \leq C \frac{V_{q_+}(x)}{\lVert V_{q_+} \rVert_{L^1(\nu)}} \mathrm{e}^{-(\iota^{\prime\prime} v)^{(1-q_+)/(1+q_+)}} \leq \frac{V_{q_+}(x)}{\lVert V_{q_+} \rVert_{L^1(\nu)}}\ \frac{\delta}{2}.
\end{align*}
Thus, 
\begin{align*}
&\Big\vert \PP^\nu(\lvert \mathbb{H}_{v,t}(f) - \mu(f) \rvert > \varepsilon) - \PP^\mu\Big(\Big\vert \frac{1}{t}\int_0^t f(X_s) \diff{s} - \mu(f)  \Big\vert > \varepsilon\Big) \Big\vert\\
&\quad \leq \int_{\R^d}\Big\vert \PP^x(\lvert \mathbb{H}_{v,t}(f) - \mu(f) \rvert > \varepsilon) - \PP^\mu\Big(\Big\vert \frac{1}{t}\int_0^t f(X_s) \diff{s} - \mu(f)  \Big\vert >\varepsilon \Big) \Big\vert \, \nu(\diff{x})  \leq \frac{\delta}{2}.
\end{align*} 
Consequently, using $t \geq \Psi(\varepsilon,\delta/2)$ and \eqref{eq:pac}, it follows by the triangle inequality that $\PP^\nu(\lvert \mathbb{H}_{v,t}(f) - \mu(f) \rvert > \varepsilon) \leq \delta.$
\end{proof}

\subsection{Discrete observations} \label{sec:discrete}
We now derive concentration inequalities for discrete observations from our continuous observation results by using the approximation strategy from \cite{galt13}. 
In \cite{galt13}, only bounded functions $f$ and  scalar exponentially ergodic diffusions in the quite strong regime \ref{cond:drift} with $q= -1$ are considered, which in particular implies sub-Gaussian tails of the invariant density. 
We demonstrate how this can be extended to the multivariate case for unbounded functions $f$ under less restrictive ergodicity assumptions.
For this purpose, the following technical key result from \cite{galt13} is of central importance.

\begin{lemma}{\cite[Proposition A.1]{galt13}} \label{lem:galt}
Let $(\Omega,\mathcal{F},(\mathcal{F}_j)_{j=1,\ldots,n},\PP)$ be a filtered probability space and $(\mathcal{X}_j)_{j=1,\ldots,n}$ be a random vector such that, for all $j \in \{1,\ldots,n\}$, $\mathcal{X}_j$ is $\mathcal{F}_j$-measurable and in $L^p(\PP)$ for some $p \geq 2$. Then, for 
\[b_{j,n}(p) \coloneqq \Big(\E\Big[\Big(\lvert \mathcal{X}_j \rvert \sum_{k=j}^n \lvert \E[\mathcal{X}_k \vert \mathcal{F}_j] \rvert \Big)^{p/2} \Big] \Big)^{2/p}, \quad j=1,\ldots,n,\]
we have 
\[\Big\Vert \sum_{j=1}^n \mathcal{X}_j \Big\Vert_{L^p(\PP)} \leq \Big(2p\sum_{j=1}^n b_{j,n}(p) \Big)^{1/2}.\]
\end{lemma}

Let $\Delta = \Delta_n \in (0,1]$ be some fixed sampling distance, and suppose that we have partial observations $(X_{\Delta k})_{k=1,\ldots,n}$ of the subexponentially ergodic diffusion process $\X$ satisfying the coefficient assumptions from Section \ref{sec:basics}.
Recall that
\[\mathbb{G}_{n,\Delta}(f)=\frac{1}{\sqrt{n\Delta}} \sum_{k=1}^n f(X_{k\Delta}) \Delta\]
denotes the discretized version of the scaled additive functional $\mathbb{G}_{n\Delta}(f)$. 
Then, for fixed $f = \tilde{f} - \mu(\tilde{f})$, we may write
\begin{equation}\label{eq:approx_discrete}
\mathbb{G}_{n,\Delta}(f) = \mathbb{G}_{n\Delta}(f) + \frac{1}{\sqrt{n\Delta}}\mathbb{A}_{n,\Delta},
\end{equation}
with discretization error
\[\mathbb{A}_{n,\Delta} \coloneqq \sum_{k=1}^n \int_{(k-1)\Delta}^{k\Delta} (\tilde{f}(X_{k\Delta}) - \tilde f(X_t)) \diff{t}.\]
With our results from Section \ref{sec:conc}, it is now clear that we must analyze the concentration of $\mathbb{A}_{n,\Delta}$ around $0$ to obtain concentration inequalities for the discrete additive functional $\mathbb{G}_{n,\Delta}(f)$. 
To do so for unbounded functionals $\tilde{f}$, we exploit polynomial $f$-norm convergence from Proposition \ref{prop:f_conv}.

\begin{theorem} \label{theo:discrete_conc}
Let $\eta_1,\eta_2, \eta_3 \geq 0$ and $\tilde{f} \in \mathcal{G}(\eta_1, \mathfrak{L}) \cap \mathcal{W}^{2,p}_{\mathrm{loc}}(\R^d)$, $p \geq d$, with $\nabla \tilde{f} \in L_{\mathrm{loc}}^{2d}(\R^d)$ such that $\lVert \nabla \tilde{f}(x) \rVert \lesssim 1+ \lVert x \rVert^{\eta_2}$ and, for all $i,j=1,\ldots,n$, $\lvert \uppartial_{x_i,x_j} \tilde{f}(x)\rvert \lesssim 1 + \lVert x \rVert^{\eta_3}$. 
Define $\alpha = \alpha(q^\prime,\eta_2,\eta_3) \coloneqq (q^\prime + \eta_2) \vee \eta_3$. In case $q > -1$, let $\tilde{\gamma} > 1+q$, $r > 1$ such that $\tilde\gamma - (1+q) > r( \alpha \vee (1+q)/(r-1))$.
If $q = -1$, set $\tilde{\gamma} = \alpha$.
Then, for $f = \tilde{f} - \mu(\tilde{f})$, there exists a constant $\mathfrak{D}$ that is independent of $n,p,\Delta$ such that, for any $p \geq 2$,
\[\left\Vert \mathbb{G}_{n,\Delta}(f) \right\Vert_{L^p(\PP^\mu)} \leq \mathfrak{D} \Big(\sqrt{n}\Delta^{3/2} + \Delta p^{\frac{\max\{(\tilde\gamma + 2\alpha +1-q_+)/2,\eta_2+1 -q_+\}}{1-q_+}} + p^{\frac{1}{2} + \frac{\eta + q^\prime + q +1}{1-q_+}} \Big) \eqqcolon \Phi(n,\Delta,p).\]
Consequently,
\[\PP^\mu(\lvert \mathbb{G}_{n,\Delta}(f) \rvert > \mathrm{e}\Phi(n,\Delta,u)) \leq \mathrm{e}^{-u}, \quad u \geq 2.\]
\end{theorem}
\begin{proof} 
Let us write $t_k = k\Delta$ and $a \lesssim b$ if $a \leq Cb$ for some constant $C$ independent of $p,n,\Delta$. 
The It\={o}--Krylov formula gives for $t \in [0, t_k]$
\begin{align*}
\tilde{f}(X_{k\Delta}) - \tilde{f}(X_t) &= \int_t^{t_k} L\tilde{f}(X_s) \diff{s} + \int_t^{t_k} \nabla \tilde{f}(X_s)^\top \sigma(X_s) \diff{W_s}\\
&= \mu(L\tilde{f})(t - t_k) + \Phi_k(t) + \omega_k(t),
\end{align*}
where $\Phi_k(t) \coloneqq \int_t^{t_k} \phi(X_s) \diff{s}$ for $\phi(y) = L\tilde{f}(y) - \mu(L\tilde{f})$ and $\omega_k(t) \coloneqq \int_t^{t_k} \nabla \tilde{f}(X_s)^\top \sigma(X_s) \diff{W_s}$. 
Thus, setting $\mathcal{X}_k = \int_{t_{k-1}}^{t_k} \Phi_k(t) \diff{t}$ and $\chi_k = \int_{t_{k-1}}^{t_k} \omega_k(t) \diff{t}$, we have 
\begin{equation}\label{eq:decomp}
\mathbb{A}_{n,\Delta} =  \mu(L\tilde{f}) \frac{n \Delta^2}{2} + \sum_{k=1}^n \mathcal{X}_k + \sum_{k=1}^n \chi_k.
\end{equation}
The polynomial bounds on the gradient and the Hessian of $\tilde{f}$ together with $\sup_{x \in \R^d} \lVert \sigma(x) \rVert < \infty$ and $\lVert b(x) \rVert \lesssim 1 + \lVert x \rVert^{q^\prime}$ imply that $\lvert L\tilde{f}(x) \rvert \lesssim 1 + \lVert x \rVert^{\alpha}$ for $\alpha \coloneqq (q^\prime + \eta_2) \vee \eta_3$. 
Suppose first $q > -1$, and let $\tilde\gamma > 1+q$ and $r > 1$ such that $\tilde\gamma - (1+q) > r( \alpha \vee (1+q)/(r-1))$.
This implies that $(\tilde\gamma - (1+q))/r >\alpha$ and $(\tilde\gamma -(1+q))/(s(1+q)) > 1$ for $s = r/(r-1)$. 
Thus, if we choose the inverse Young functions $\Psi_1(x) = (sx)^{1/s}$ and $\Psi_2(x) =(rx)^{1/r}$, it follows for $f_{\tilde\gamma,q}(x) = \lVert x \rVert^{\tilde\gamma - (1+q)}$ that $\lvert L\tilde{f} \rvert \lesssim \bm 1 \vee \Psi_2 \circ f_{\tilde\gamma,q}$.
Proposition \ref{prop:f_conv} then yields 
\begin{equation}\label{eq:f_conv}
\lVert \PP^x(X_t \in \cdot) - \mu \rVert_{\bm 1 \vee \Psi_2 \circ f_{\tilde\gamma,q}} \lesssim \tilde{V}_{\tilde\gamma}(x) (1+t)^{-(\tilde\gamma - (1+q))/(s(1+q))}, \quad x \in \R^d, t \geq 0.
\end{equation}
Let $\mathbb{F} = (\mathcal{F}_t)_{t \geq 0}$  be the natural filtration of $(X_t)_{t \geq 0}$. Then, using the Markov property and \eqref{eq:f_conv}, we obtain for $t > u$
\begin{align*}
\lvert \E^\mu[\phi(X_t) \vert \cF_u] \rvert &= \lvert \E^{X_u}[\phi(X_{t-u})] \rvert = \lvert \E^{X_u}[L\tilde{f}(X_{t-u})] - \mu(L\tilde{f})\rvert\\
&\lesssim \lVert \PP^{X_u}(X_{t-u} \in \cdot) - \mu \rVert_{\one \vee \Psi_2 \circ f_{\tilde\gamma,q}} \lesssim \tilde{V}_{\tilde\gamma}(X_{u}) (1+(t-u))^{-(\tilde\gamma - (1+q))/(s(1+q))}.
\end{align*}
For $k > j$, this gives 
\begin{align*}
\E^\mu[\mathcal{X}_k \vert \mathcal{F}_{t_j}] &\lesssim \tilde{V}_{\tilde\gamma}(X_{t_j}) \int_{t_{k-1}}^{t_k} \int_{t}^{t_k} (1+(u-t_j))^{-(\tilde\gamma - (1+q))/(s(1+q))} \diff{u}\diff{t}\\
&\leq \tilde{V}_{\tilde\gamma}(X_{t_j}) \Delta^2 (1+(k - 1 -j)\Delta)^{-(\tilde\gamma - (1+q))/(s(1+q))}.
\end{align*}
Hence, for $j < n$,
\[\sum_{k=j+1}^n \lvert \E^\mu[\mathcal{X}_k \vert \mathcal{F}_{t_j}] \rvert \lesssim \tilde{V}_{\tilde\gamma}(X_{t_j})\Delta^2 \int_0^\infty (1+ \Delta t)^{-(\tilde\gamma-(1+q))/(s(1+q))} \diff{t} = \tilde{V}_{\tilde\gamma}(X_{t_j}) \Delta \frac{s(1+q)}{\tilde\gamma - (1+q)(1+s)},\]
where we used that $(\tilde\gamma - (1+q))/(s(1+q)) > 1$. 
Consequently, letting $b_{j,n}(p)$ be the functional from Lemma \ref{lem:galt}, it follows for $j< n$ from the Cauchy--Schwarz inequality, stationarity  and Lemma \ref{lem:l_p}
\begin{align*} 
b_{j,n}(p) &\leq \lVert \mathcal{X}_j \rVert^2_{L^{2p}(\PP^\mu)} \Big\Vert \sum_{k=j+1}^n \lvert \E^\mu[\mathcal{X}_k \vert \mathcal{F}_{t_j}] \rvert \Big\Vert_{L^p(\PP^\mu)} \lesssim \Delta \lVert \mathcal{X}_j \rVert^2_{L^{2p}(\PP^\mu)} \lVert \tilde{V}_{\tilde\gamma}(X_0) \rVert_{L^p(\PP^\mu)} \lesssim \Delta p^{\frac{\tilde\gamma}{1-q_+}}\lVert \mathcal{X}_j \rVert^2_{L^p(\PP^\mu)}.
\end{align*}
In case $q = -1$, we simply observe that Proposition \ref{prop:f_conv} implies that there exists $\beta > 0$ such that $\lVert \PP^x(X_t \in \cdot) - \mu \rVert_{\tilde{V}_\alpha} \lesssim \tilde{V}_\alpha(x)\exp(-\beta t)$ and hence, proceeding as above, we end up with 
\[b_{j,n}(p) \lesssim \Delta \lVert \mathcal{X}_j \rVert^2_{L^{2p}(\PP^\mu)} \lVert \tilde{V}_\alpha(X_0)\rVert_{L^p(\PP^\mu)} \lesssim \Delta p^{\frac{\alpha}{1-q_+}}\lVert \mathcal{X}_j \rVert^2_{L^p(\PP^\mu)}. \]
Now, stationarity under $\PP^\mu$, Hölder's inequality together with Fubini and  Lemma \ref{lem:l_p} yield 
\begin{align*} 
\lVert \mathcal{X}_j \rVert_{L^p(\mu)}^p &= \E^\mu\Big[\Big(\int_0^\Delta \int_t^\Delta \phi(X_s) \diff{s} \diff{t}\Big)^p\Big] \leq \Delta^{2(p-1)} \int_0^\Delta \int_{t}^\Delta \E^\mu\big[\lvert \phi(X_s) \rvert^p\big] \diff{s} \diff{t}\\
&= \Delta^{2p} \lVert \phi(X_0) \rVert_{L^p(\PP^\mu)}^p \leq c^p \Delta^{2p} p^{p\alpha/(1-q_+)},
\end{align*}
for some constant $c > 0$. 
Thus, we obtain 
\[b_{j,n}(p) \lesssim \Delta^3 p^{(\tilde\gamma+ 2\alpha)/(1-q_+)},\]
and hence by Lemma \ref{lem:galt} 
\begin{equation} \label{eq:boundsum1}
\Big\Vert \sum_{k=1}^n \mathcal{X}_k \Big\Vert_{L^p(\PP^\mu)} \lesssim \sqrt{n \Delta^3} p^{(\tilde\gamma + 2\alpha +1-q_+)/(2(1-q))}.
\end{equation}
Let us now treat $\sum_{k=1}^n \chi_k$. As in the proof of Lemma \ref{lem:mart_approx}, we obtain by the Burkholder--Davis--Gundy inequality, \cite[Proposition 4.2]{barlow82} and Lemma \ref{lem:l_p} that 
\[\E^\mu[\lvert \omega_k \rvert^p] \leq c^p p^{p/2} \lambda_+^p \Delta^{p/2} \lVert \nabla \tilde{f} \rVert_{L^p(\mu)}^p \leq C(\eta_2)^p \lambda_+^p \Delta^{p/2} p^{p/2 + p\eta_2/(1-q_+)}.\]
Therefore, with Hölder's inequality,
\[
\E^\mu[\lvert \chi_k \rvert^p] \leq \Delta^{p-1} \int_{t_{k-1}}^{t_k} \E^\mu[\lvert \omega_k(t)\rvert^p] \diff{t} \leq C(\eta_2)^p \lambda_+^p \Delta^{3p/2} p^{p/2 + p\eta_2/(1-q_+)}.
\]
Let $b^\chi_{n,p}$ be the functional from Lemma \ref{lem:galt} with respect to $(\chi_k)_{k=1,\ldots,n}$ and $(\mathcal{F}_{t_k})_{k=1,\ldots,n}$, and note that, for $k > j$ and $t \in [t_j,t_k]$, $\E^\mu[\omega_k(t) \vert \mathcal{F}_{t_j}] = 0$ since $(\int_0^t\nabla \tilde{f}(X_s)^\top \sigma(X_s) \diff{W_s})_{t \geq 0}$ is an $\mathbb{F}$-martingale. 
Thus,  
\[b_{j,n}^\chi(p) = \E^\mu[\lvert \chi_j \rvert^p]^{2/p} \lesssim \Delta^3 p^{(2\eta_2 + (1-q_+))/(1-q_+)}.\]
Consequently, by Lemma \ref{lem:galt},
\begin{equation} \label{eq:boundsum2}
\Big\Vert \sum_{k=1}^n \chi_k \Big\Vert_{L^p(\PP^\mu)} \lesssim \sqrt{n}\Delta^{3/2} p^{(\eta_2 + 1-q_+)/(1-q_+)}.
\end{equation}
Taking into account that $\mu(\lvert L\tilde{f} \rvert) < \infty$, \eqref{eq:decomp}, \eqref{eq:boundsum1} and \eqref{eq:boundsum2} imply that 
\[\Big\Vert \frac{1}{\sqrt{n\Delta}} \mathbb{A}_{n\Delta} \Big\Vert_{L^p(\PP^\mu)} \lesssim \sqrt{n}\Delta^{3/2} + \Delta p^{\frac{\max\{(\tilde\gamma + 2\alpha +1 -q_+)/2,\eta_2+1-q_+\}}{1-q_+}}.\]
Plugging this bound into \eqref{eq:approx_discrete} and using Theorem \ref{prop:conc}, it follows that there exists some constant $\mathfrak{D}$ that is independent of $n,p,\Delta$ such that, for $p \geq 1$,
\[\Big\Vert \mathbb{G}_{n,\Delta}(f) \Big\Vert_{L^p(\PP^\mu)} \leq \mathfrak{D} \Big(\sqrt{n}\Delta^{3/2} + \Delta p^{\frac{\max\{(\tilde\gamma + 2\alpha +1-q_+)/2,\eta_2+1-q_+\}}{1-q_+}} + p^{\frac{1}{2} + \frac{\eta + q^\prime + q +1}{1-q_+}} \Big).\]
Markov's inequality now yields the asserted concentration inequality.
\end{proof}

\paragraph{PAC bounds}
Similarly to Corollary \ref{coro:pac} and Corollary \ref{coro:burnin}, we can derive PAC bounds for the discrete ergodic average and its burn-in version. 
The proof is identical and therefore omitted.

\begin{corollary} \label{coro:pac_discrete}
Let $\eta_1,\eta_2,\eta_3 \geq 0$ and $f \in \mathcal{G}(\eta_1, \mathfrak{L}) \cap \mathcal{W}^{2,p}_{\mathrm{loc}}(\R^d)$, $p \geq d$, with $\nabla f \in L_{\mathrm{loc}}^{2d}(\R^d)$ such that $\lVert \nabla f(x) \rVert \lesssim 1 + \lVert x \rVert^{\eta_2}$ and, for all $i,j=1,\ldots,d$, $\lvert \uppartial_{x_i,x_j} f(x) \rvert \lesssim 1 + \lVert x \rVert^{\eta_3}$. 
Define $\alpha,\tilde{\gamma}$  as in Theorem \ref{theo:discrete_conc}, and denote
\[\varrho = \varrho(\alpha,\eta_2,\tilde{\gamma},q) \coloneqq \frac{\max\{(\tilde{\gamma} + 2\alpha +1-q_+)/2,\eta_2 + 1 -q_+ \}}{1-q_+}\]
and 
\[\tilde{\varsigma} = \tilde{\varsigma}(\eta_1,q,q^\prime) \coloneqq \frac{1}{2} + \frac{\eta + q^\prime + q +1}{1-q_+}. \]
For $\varepsilon > 0$, $\delta \in (0,\mathrm{e}^{-2})$, suppose that $\Delta < \varepsilon/(3\mathrm{e}\mathfrak{D})$ and 
\[n \geq \Psi(\Delta,\varepsilon,\delta) \coloneqq \frac{1}{\Delta} \left(\frac{3\mathrm{e}\mathfrak{D} \max\big\{\Delta (\log(1/\delta))^{\varrho}, (\log(1/\delta))^{\tilde\varsigma} \big\}}{\varepsilon} \right)^2.\]
Then, 
\[\PP^\mu\Big(\Big \vert \frac{1}{n} \sum_{k=1}^n f(X_{k\Delta}) - \mu(f) \Big\vert \leq \varepsilon \Big) \geq 1- \delta.\]
Moreover, let the discrete burn-in estimator be given by
\[\mathbb{H}_{m,n,\Delta}(f) \coloneqq \frac{1}{\sqrt{n\Delta}} \mathbb{G}_{n,\Delta}(f) \circ \theta_{m\Delta} = \frac{1}{n} \sum_{k=m+1}^{n+m} f(X_{k\Delta}).\]
Then, given the constants $\iota^{\prime\prime}, C$ from Corollary \ref{coro:burnin} and some  initial distribution $\nu$ such that $V_{q_+} \in L^1(\nu)$, for any $n \geq \Psi(\Delta,\varepsilon,\delta/2)$ and burn-in length $m \geq 1\vee \Delta^{-1}(\log(2C/\delta))^{(1+q_+)/(1-q_+)}/\iota^{\prime\prime}$, it holds that 
\[\PP^\nu\Big(\vert \mathbb{H}_{m,n,\Delta}(f) - \mu(f)\vert \leq \varepsilon \Big) \geq 1- \delta.\]
\end{corollary}

\section{Applications}
We now demonstrate the usefulness of our probabilistic results in two concrete applications.
While exponential inequalities are important for a multitude of statistical problems (e.g., in the context of adaptive nonparametric estimation or for the verification of uniform convergence results), we will focus in Section \ref{sec:lasso} on the analysis of a high-dimensional  diffusion model under sparsity constraints, which in particular necessitates the use of inequalities for \emph{unbounded} functions. 
Specifically, we will see that Theorem \ref{prop:conc} allows us to derive non-asymptotic error bounds for penalised estimators, which, to the best of our knowledge, are so far only available for Ornstein--Uhlenbeck processes. 
In Section \ref{sec:mcmc}, we use our discrete concentration results from Section \ref{sec:discrete} to derive explicit convergence guarantees for an MCMC algorithm designed to sample from target densities with subexponential tails.

\subsection{Lasso estimation for parametrized drift coefficients}\label{sec:lasso}

As opposed to the now very well understood high-dimensional discrete models (cf., e.g., \cite{buvdg11} or \cite{wain19}), for which a wealth of estimation algorithms including corresponding theoretical results are available, there are still few in-depth studies of estimation problems for high-dimensional continuous-time processes.
Important references in this context are \cite{gaiffas19} and \cite{ciolek20}, who investigate drift estimation in a high-dimensional Ornstein--Uhlenbeck (OU) model under sparsity constraints. 
A remarkable feature is that the restricted eigenvalue property, which usually has to be verified explicitly in discrete models such as linear regression, is already implied by the ergodicity in the specified diffusion model. 
This finding is based on the use of sufficiently sharp probabilistic tools in the form of concentration inequalities suited to the model: 
while \cite{gaiffas19} provide a proof based on functional inequalities allowing to cover only the reversible case, \cite{ciolek20} use Malliavin calculus methods to show that the restricted eigenvalue property is satisfied in the general ergodic OU case.
At the same time, they point out (cf.~their Remark 4.4) that other mathematical methods are needed for proving such concentration phenomena in more general diffusion models.

Motivated by the considerations in \cite{pokern2009}, we outline in this section how our results from Section \ref{sec:conti} can be used to study more general high-dimensional diffusion models.
Suppose that the data $X^T=(X_t)_{0\le t\le T}$ has been generated by the following It\={o} SDE,
\begin{equation}\label{eq:b0}
\d X_t=b_0(X_t)\d t+\sigma_0(X_t)\d W_t,
\end{equation}
$W=(W_t)_{t\ge0}$ a standard $d$-dimensional Brownian motion. 
The diffusion matrix $\sigma_0$ is assumed to be known and we wish to estimate the drift vector $b_0$.
Suppose that both $\sigma_0$ and $b_0$ are globally Lipschitz, that $\sigma_0$ is bounded, that $a_0 \coloneqq \sigma_0 \sigma_0^\top$ is uniformly elliptic, i.e., 
\[
\exists \lambda_-,\lambda_+>0\ \forall x,\eta\in\mathbb R^d: \quad \lambda_-\lVert\eta\rVert^2 \leq \left\langle \eta,a_0(x)\eta\right\rangle\le \lambda_+\lVert\eta\rVert^2,
\]
and that the drift condition 
\begin{enumerate}[label = ($\mathscr{L}(q)$), ref = ($\mathscr{L}(q)$), leftmargin=*]
\item \label{cond:drift_lasso} there exists $M_0, \mathfrak{r} > 0$ such that 
\[ \forall \lVert x \rVert > M_0: \, \langle b_0(x), x\slash \lVert x \rVert\rangle \leq -\mathfrak{r} \lVert x \rVert^{-q}\]
\end{enumerate}
is satisfied for some $q \in [-1,1)$, such that the process falls into the ergodic framework of Section \ref{sec:basics}.
Denote by $\mu_0$ and $\rho_0$ its invariant measure and the corresponding invariant density, respectively. We further assume that $X^T$ is the stationary solution, i.e., $X_0\sim \mu_0$, and denote $\PP \coloneqq \PP^\mu$.

Denote by $\P_{b}$  the law of $Y^T$, where $Y^T=(Y_t)_{0\le t\le T}$ is the strong solution to the SDE $\d Y_t = b(Y_t)\d t + \sigma_0(Y_t)\d W_t$, $Y_0=X_0$. 
Then, the Radon--Nikodym derivative of $\P_b$ with respect to $\P_0$ is given as 
\[\frac{d \P_b}{d \P_0}(X^T)= \exp\left(-\frac{1}{2}\int_0^T b^\top(X_t)a_0^{-1}(X_t)b(X_t)\d t+ \int_0^T b^\top(X_t)a_0^{-1}(X_t)\d X_t\right)\]
(see \cite[Section 7.6.4]{liptser2001}).
Given the data $X^T$, one can derive the negative of the $\log$ likelihood functional for the unknown drift $b$.
Up to an irrelevant constant, this functional is given by
\begin{equation}\label{eqa215}
\mathcal L_T(b)=
\frac1T\int_0^T\left(b^\top(X_t)a_0^{-1}(X_t)b(X_t)\d t-2 b^\top(X_t)a_0^{-1}(X_t)\d X_t\right).
\end{equation}
The $\log$ likelihood function \eqref{eqa215} for $b$ is unbounded below in general if the data is finite, $T<\infty$.
However, letting $T\to\infty$, \eqref{eqa215} tends to a functional whose unique minimizer is $b_0$. More precisely, it is shown in Lemma 6.1 in \cite{pokern2009} that $\mathcal L_T(b)$ converges a.s.~towards the functional 
\[
\mathcal L_\infty(b)= \int_{\R^d}\left(b^\top(x)a_0^{-1}(x)\left(b(x)-2b_0(x)\right)\right)\rho_0(x)\d x.
\]
In order to regularize \eqref{eqa215}, \cite{pokern2009} suggest to assume a \emph{parametric} structure of the drift coefficient. 
For the class of generalised OU processes fulfilling the linear SDE 
\begin{equation}\label{gen:ou}
\d X_t= -\boldsymbol A X_t\d t+\sigma\d W_t, \quad t\geq0,
\end{equation}
$\bm A$ and $\sigma$ some $d\times d$-matrices and $W$ a $d$-dimensional Brownian motion, this assumption is obviously satisfied.
A more general, but still treatable class of processes is obtained as follows: 
Given a system $(\psi_j)_{1\leq j\leq N}$ of Lipschitz continuous basis functions $\psi_j\colon\R^d\to \R^d$, introduce
\begin{equation*}
\mathcal V\coloneqq \Big\{b_\theta(\cdot) = \sum_{j=1}^N \theta_j\psi_j(\cdot),\ \theta\in\R^N\Big\}.
\end{equation*}
Let $\bm{\psi}(\cdot) \coloneqq (\psi_1(\cdot),\ldots,\psi_N(\cdot))$ be the dictionary matrix and $\bm{\Psi}(x) \coloneqq (\sigma_0^{-1}(x)\bm{\psi}(x))^\top \sigma_0^{-1}(x)\bm\psi(x)$ for $x \in \R^d$. Let us also define the matrices
\[\Op_T\coloneqq \frac{1}{T} \int_0^T \bm{\Psi}(X_s) \diff{s} = (\overline \psi_{ij,T})_{1\leq i,j\leq N}
\quad\text{and}\quad \Op_\infty\coloneqq \E[\bm{\Psi}(X_0)] = (\overline\psi_{ij,\infty})_{1\le i,j\le N}
\]
with entries 
\begin{equation*}\label{def:psi}
\begin{split}
\overline \psi_{ij,T}&\coloneqq\frac1T\int_0^T\left\langle \psi_i(X_s), a_0^{-1}(X_s) \psi_j(X_s)\right\rangle\d s,\\
\overline\psi_{ij,\infty}&\coloneqq \int_{\R^d}\left\langle\psi_i(x),a_0^{-1}(x)\psi_j(x)\right\rangle \rho_0(x)\d x,
\quad i,j=1,\ldots,N.
\end{split}
\end{equation*}

We impose the following assumptions on the dictionary:
\begin{enumerate}[label = ($\mathscr{L}$\arabic*), ref= ($\mathscr{L}$\arabic*), leftmargin = *]
\item \label{ass:drift_growth} 
There exist $\mathfrak{L} > 0$ and $\eta \in [0,1]$ such that the maximal eigenvalue of $\bm{\Psi}(x)$ satisfies
\[\lambda_{\max}(\bm{\Psi}(x)) \leq \mathfrak{L}(1+ \lVert x \rVert^{2\eta}), \quad x \in \R^d;\] 
\item \label{ass:nonhyper} 
the random matrix $\Op_T$ is positive definite $\PP$-a.s.
\end{enumerate}
Assumption \ref{ass:drift_growth} allows a maximal polynomial drift of order $\eta$ of the basis functions, where $\eta \in [0,1]$ is  consistent with their assumed Lipschitz continuity. 
Assumption \ref{ass:nonhyper} is a necessary technical condition on the positive semidefinite matrix $\Op_T$ that we need for the penalized MLE to be well-defined. It can be verified given sufficient smoothness of the dictionary, see Example \ref{ex:lasso}. 
Moreover, since by stationarity  $\Op_\infty = \E[\Op_T]$, \ref{ass:nonhyper} implies that $\Op_\infty$ is positive definite. Therefore, if we denote the minimal eigenvalue of $\Op_\infty$ by $\lambda_{\min}(\Op_\infty) \eqqcolon \mathfrak{e}_\infty$, then $\mathfrak{e}_\infty > 0$. 
Let us also set $\mathfrak D_\infty\coloneqq \max_{i=1,\ldots,N}\overline\psi_{ii,\infty}$.

We now give an  example of a dictionary that satisfies the above assumptions and can be used to model drifts satisfying the drift condition \ref{cond:drift_lasso}.

\begin{example}\label{ex:lasso}
Let 
\[ E_i = \one_{1 + \llfloor (i- d^2\llfloor i \slash d^2 \rrfloor)\slash d \rrfloor, 1 + (i-1) \bmod d},  \quad i = 1,\ldots, nd^2,\]
where $\one_{k,l}$ is the $d\times d$ matrix whose $(k,l)$-th entry is $1$ and all other entries are $0$, and $\llfloor x \rrfloor = \max\{z \in \Z: z < x\}$. 
Set then, for $\tilde{q}_i \in [-1,1)$ and $\tilde{\alpha}_i > 0$,
\[\psi_i(x) =  E_i x (\tilde{\alpha}_i + \lVert x \rVert)^{-(\tilde{q}_i+1)}, \quad \text{ where } \tilde{q}_i = \tilde{q}_j \text{ and } \tilde{\alpha}_i = \tilde{\alpha}_j \text{ if } \llfloor i/d^2 \rrfloor = \llfloor j/d^2 \rrfloor,\]
which is nothing else but saying that any $b \in \mathcal{V}$ can be written as 
\[b_\theta(x) \coloneqq \sum_{i=1}^{N} \theta_i \psi_i(x) = \sum_{i=1}^{n} A_i(\theta) x(\alpha_i + \lVert x \rVert)^{-(q_i + 1)}, \quad x \in \R^d,\]
where $N = nd^2$,
\[(A_i(\theta))_{k,l} = \theta_{(i-1)d^2 + (k-1)d + l}, \quad i =1,\ldots,n \text{ and } k,l = 1,\ldots, d,\]
and $q_i = \tilde{q}_{1 + d^2\llfloor i/ d^2 \rrfloor}, \alpha_i = \tilde{\alpha}_{1 + d^2\llfloor i/ d^2 \rrfloor}$. Suppose that $q_i < q_j$ for $i > j$, the matrices $A_i(\theta_0)$ corresponding to the true value $\theta_0$ are symmetric, and that there exists $k_0 \in \{1,\ldots,n\}$ such that $\lambda_{\max}(A_{k_0}(\theta_0)) < 0$ and, for all $k_0 < k \leq n$, it holds that $\lambda_{\max}(A_k(\theta_0)) = 0$. Then, it follows from the Courant--Fischer theorem that, for any $x \neq 0$,
\begin{align*}
\langle b_{\theta_0}(x), x\slash \lVert x \rVert \rangle &= \sum_{i=1}^n \lVert x \rVert ( \alpha_i + \lVert x\rVert)^{-(1+q_i)} \langle x\slash \lVert x \rVert, A_i(\theta_0)x\slash \lVert x \rVert \rangle\\
&\leq \sum_{i=1}^{k_0} \lVert x \rVert ( \alpha_i + \lVert x\rVert)^{-(1+q_i)} \lambda_{\max}(A_i(\theta_0)).
\end{align*}
This implies that there exists $M_0, c > 0$ such that, for $\mathfrak{r} = -c\lambda_{\max}(A_{k_0}(\theta_0)) > 0$, the drift condition \ref{cond:drift_lasso} is satisfied for $q = q_{k_0}$. 
Let $\mu$ be the invariant distribution of the associated  diffusion $\X$. Also note that \ref{ass:drift_growth} holds for $\eta = (-q_1)_+$. 

To see that \ref{ass:nonhyper} is satisfied, note first that, for any $\theta \neq 0$, there exists some $j \in \{1,\ldots,d\}$ such that $x \mapsto (b_\theta(x))_j$ is analytic and not identical to zero on $\R^d \setminus \{0\}$. Consequently, $(b_\theta(\cdot)_j)^{-1}(\{0\}) \setminus\{0\} = \{x \neq 0: (\bm{\psi}(x)\theta)_j = 0\}$ is contained in a countable union of smooth manifolds of dimension $d-1$, i.e., in a countable union of  smooth hypersurfaces. 
Assume now that $\Op_T$ is not positive definite a.s. Since $\theta^\top \Op_T \theta = \int_0^T \lVert (\sigma_0^{-1}\bm{\psi})(X_s)\theta\rVert^2 \diff{s}$ and, moreover, the matrix is positive semidefinite, the paths of $\X$ are continuous and $\PP^\mu(X_0 = 0) = 0$, this implies that there exists a measurable set $\Omega_0 \subset \{X_0 \neq 0\}$ with $\PP^\mu(\Omega_0) > 0$ such that, for any $\omega \in \Omega_0$, the whole path $(X_s(\omega))_{s \in [0, T]}$ is contained in $(b_\theta(\cdot)_j)^{-1}(\{0\})$ for some $\theta \neq 0$ and $j \in \{1,\ldots,d\}$. It follows from above that on $\Omega_0$, the process stays in some smooth hypersurface for a strictly positive amount of time. Such path behaviour is however impossible a.s.\ for an elliptic diffusion process.  Thus, $\Op_T$ must be positive definite $\PP$-a.s.
\end{example}

Note that the above example includes the OU models investigated in \cite{gaiffas19,ciolek20} as a special case. Under $\P_{b_0}$, the above parametrisation yields the functional
\begin{align*}
\mathcal L_T(\theta)&=
-\frac1T\left(2\int_0^T\big(\sigma_0^{-1}b_\theta\big)^\top(X_t)\d W_t - \int_0^T\left\|\sigma_0^{-1}\left(b_\theta-b_{\theta_0}\right)(X_t)\right\|^2\d t +
\int_0^T\left\|\big(\sigma_0^{-1}b_{\theta_0}\big)(X_t)\right\|^2\d t\right)\\
&= \theta^\top\Op_T\theta-2\theta^\top\overline h,
\end{align*} 
$\overline h$ denoting the vector with components
\[\overline h_i=\frac1T\int_0^T\left\langle\psi_i(X_s),a_0^{-1}(X_s)\d X_s\right\rangle, \quad i=1,\ldots,N.\]
Using almost sure positive definiteness of $\Op_T$, it follows that on a set of full $\PP$-measure, the MLE is the unique minimizer of $\mathcal L_T(\cdot)$, given by
\[
\hat\theta_{\mathrm{MLE}} \coloneqq \Op{}^{-1}_T\overline h.
\]
While this approach yields a well-defined estimator, the MLE will perform quite inaccurately in high-dimensional settings.

Our concern is to investigate the estimation of $b_\theta$ in the large $N$/large $T$ regime.
More precisely, we want to study the statistical properties of penalized estimators $\hat\theta_T$, defined as 
\begin{equation}\label{lasse}
\widehat\theta_T= \operatorname{arg\ min}_{\theta \in \R^N} \left\{\mathcal L_T(\theta)+\lambda \|\theta\|_1\right\},
\end{equation}
$\lambda>0$ some regularisation parameter. 
Strictly speaking, since positive definiteness of $\Op_T$ holds only a.s., this estimator may only be well defined in an almost sure sense, but by an appropriate restriction of the underlying probability space we can and will assume that it is well-defined everywhere without loss of generality.
Denote
\[
\left\|\theta_1-\theta_2\right\|_{L^2}^2
\coloneqq 
\frac1T\int_0^T\left\|\sigma_0^{-1}\left(b_{\theta_1}-b_{\theta_2}\right)(X_t)\right\|^2\d t = (\theta_1 - \theta_2)^\top \Op_T (\theta_1-\theta_2), \quad \theta_1,\theta_2\in\R^N.
\]
Then, for any $\theta\in\R^N$,
\begin{equation}\label{eq:bas1}
\big\|\hat\theta_T-\theta_0\big\|_{L^2}^2\le \big\|\theta-\theta_0\big\|_{L^2}^2+\frac2T\int_0^T\left(\sigma_0^{-1}\left(b_{\hat\theta_T}-b_\theta\right)\right)^\top(X_t)\d W_t+\lambda\left(\|\theta\|_1-\big\|\hat\theta_T\big\|_1\right).
\end{equation}
In order to obtain error bounds for the Lasso estimator $\hat\theta_T$, the martingale part appearing on the rhs of \eqref{eq:bas1} needs to be controlled which is usually done by means of Bernstein's inequality for continuous martingales. 
Another important part of the derivation of error bounds is the verification of the restricted eigenvalue condition which in our setting amounts in showing that
\begin{align*}
\inf_{\theta\in \mathcal S_1(s), \eta\in\mathcal S_2(s,\theta)}
\frac{\|\theta-\eta\|_{L^2}^2}{\|\theta-\eta\|^2}\quad\text{is bounded away from 0 with high probability},
\end{align*}
where, for $\|\theta\|_0\coloneqq \sum_i \one_{\{\theta_i\ne 0\}}$, fixed $c_0>0$ and $\mathcal I_s(\theta)$ denoting a set of coordinates of $s$ largest elements of $\theta$,
\begin{align*}
\mathcal C(s,c_0)&\coloneqq 
\left\{\zeta\in\R^{N}:
\|\zeta\|_1\le(1+c_0)\left\|\zeta_{|\mathcal I_s(\zeta)}\right\|_1\right\},\\
\mathcal S_1(s)&\coloneqq \left\{\theta\in\R^N: \|\theta\|_0=s\right\} \quad\text{ and }\quad
\mathcal S_2(s,\theta)\coloneqq \left\{\eta\in\R^N: \theta-\eta\in\mathcal C(s,c_0)\right\}.
\end{align*}

To start with, we will demonstrate how our previous general developments can be used to verify these assumptions. 
In fact, our error bounds for the Lasso estimator formulated below are based on the following direct application of Theorem \ref{prop:conc}.

\begin{lemma}\label{lem:conc}
There exists a constant $\mathfrak{W} > 0$ such that, for any vectors $\zeta\in\R^N$ with $\|\zeta\|\le 1$ and $R \geq 2/\sqrt{T}$,
\begin{equation}\label{ineq:appl}
	\PP\left(\big|\zeta^\top\big(\Op_\infty - \Op_T\big)\zeta\big|>R\right)  \leq \exp\bigg(-\bigg(\frac{\sqrt{T}R}{\mathrm{e}\mathfrak{L}\mathfrak{W}} \bigg)^{\kappa(q,\eta)} \bigg),
	\quad\text{where }\kappa(q,\eta) \coloneqq \frac{2(1-q_+)}{6\eta + 2q+3-q_+}.
\end{equation}
\end{lemma}
\begin{proof}
Observe first that it suffices to prove the lemma for $\|\zeta\|=1$.
Fix any such $\zeta$ and set $\tilde{f}_\zeta(x) =\zeta^\top \bm{\Psi}(x) \zeta$ and $f_\zeta = \tilde{f}_\zeta - \mu_0(\tilde{f}_\zeta)$. 
By assumption \ref{ass:drift_growth}, we have for any $x \in \R^d$
\[\lvert \tilde{f}_\zeta(x) \rvert = \lVert \sigma_0^{-1}(x)\bm{\psi}(x)\zeta\rVert^2 \leq \lVert \sigma_0^{-1}(x) \bm{\psi}(x) \rVert^2 = \lambda_{\max}(\bm{\Psi}(x)) \leq \mathfrak{L}(1+ \lVert x \rVert^{2\eta}).\]
Moreover, using $\lVert \sigma_0(x) \rVert = \lVert \sigma_0(x)^\top \rVert$,
	\begin{align*}
		\max_{i=1,\ldots,N} \lVert \psi_i(x) \rVert &\leq \sqrt{\lambda_+} \max_{i=1,\ldots,n} \lVert \sigma_0^{-1}(x) \psi_i(x)\rVert = \sqrt{\lambda_+} \max_{i=1,\ldots,N} \lVert\sigma_{0}^{-1}(x)\bm{\psi}(x)e_i\rVert\\
		&\leq  \sqrt{\lambda_+}\lVert\sigma_{0}^{-1}(x)\bm{\psi}(x)\rVert \leq \sqrt{\mathfrak{L}\lambda_+}(1+ \lVert x \rVert^\eta),
	\end{align*}
	such that $\lVert b_{\theta_0}(x)\rVert \leq \sqrt{\mathfrak{L}\lambda_+} \lVert \theta_0 \rVert_1 (1 + \lVert x \rVert^\eta)$ follows.
Consequently, Theorem \ref{prop:conc} implies that there exists some constant $\mathfrak{W}$ independent of $\zeta$ such that 
\[
\PP\left(\big|\zeta^\top\big(\Op_\infty - \Op_T\big)\zeta\big|>R\right) = \PP\big(T^{-1/2}\lvert \mathbb{G}_T(f_\zeta) \rvert  > R \big) \leq \exp\bigg(-\bigg(\frac{\sqrt{T}R}{\mathrm{e}\mathfrak{L}\mathfrak{W}} \bigg)^{\kappa(q,\eta)} \bigg).
\]
\end{proof}

We are now ready to verify the restricted eigenvalue property and state deviation bounds for the martingale term. 

\begin{proposition}\label{prop:res_ev_con}
\begin{enumerate}[label=(\alph*)]	
\item
For any $\ep_0\in(0,1)$ and $\forall T\ge T_0(\ep_0,s,c_0, \mathfrak {LW})$, it holds that
\begin{align*}
\mathbb P\left(\inf_{\theta\in\mathcal S_1(s),\eta\in\mathcal S_2(s,\theta)}
\frac{\|\theta-\eta\|_{L^2}^2}{\|\theta-\eta\|^2}\ge \frac{\mathfrak e_\infty}{2}\right)&\ge 1-\ep_0,
\end{align*}
where 
\[
T_0(\ep_0,s,c_0,c)\coloneqq 
\left\{\log\left(21^{2s}\left(d\wedge \left(\tfrac{\e d}{2s}\right)^{2s}\right)\right) -\log\ep_0\right\}^{\frac{2}{\kappa(q,\eta)}}\cdot\frac{18^2\left(c_0+2\right)^2\e^2c^2}{\mathfrak e_\infty^2} 
\]
\item
For $s,c_0>0$, define the event
\begin{equation}\begin{split}\label{def:eve}
\mathscr E(s,c_0)&\coloneqq
\left\{\inf_{\theta-\eta\in\mathcal C(s,c_0)}\frac{\|\theta-\eta\|_{L^2}^2}{\|\theta-\eta\|^2}\ge\frac{\mathfrak e_\infty}{2}\right\}\cap 	\left\{\max_{i=1,\ldots,N}\overline\psi_{ii,T}\le \mathfrak D_\infty+\frac{\mathfrak e_\infty}{2}\right\}\\
&\qquad\cap \left\{\sup_{\theta\ne\eta\in\R^N}\frac{\frac1T\int_0^T\left(\sigma_0^{-1}\big(b_\theta-b_\eta\big)\right)^\top(X_t) \d W_t}{\|\theta-\eta\|_1}\le\frac\lambda2	\right\}.	
\end{split}
\end{equation}
Then, for any $\ep_0 \in (0,1)$, $T\ge T_0(\tfrac{\ep_0}{3},s,c_0,\mathfrak{LW})$ and
\[
\lambda \ge \sqrt{\frac{4\left(2\mathfrak D_\infty+\mathfrak e_\infty\right)}{T} \cdot \log\left(\frac{6N}{\ep_0}\right)},
\]	
it holds $\mathbb P(\mathscr E(s,c_0))\ge 1-\ep_0$.
\end{enumerate}
\end{proposition}
\begin{proof}
Introduce $\mathcal K(s)\coloneqq \left\{\zeta\in\mathbb R^N\setminus\{0\}: \|\zeta\|_0\le s \right\}$.
Using Lemmata F.1 and F.3 of \cite{basu2015}, it follows
			\[\sup_{\zeta\in\mathcal C(s,c_0)}\frac{\zeta^\top\big(\Op_\infty-\Op_T\big)\zeta}{\|\zeta\|^2}\le 3\left(c_0+2\right)\sup_{\zeta\in\mathcal K(2s)}\frac{\zeta^\top\big(\Op_\infty-\Op_T\big)\zeta}{\|\zeta\|^2 }\]			
Furthermore, for any subset $E\subset\mathbb R^N$,
\[
\inf_{\zeta\in E}\frac{\|\zeta\|_{L^2}^2}{\|\zeta\|^2}=\inf_{\zeta\in E:\|\zeta\|\le 1}\zeta^\top\Op_T\zeta
\]and for $\zeta \neq 0$,
\begin{align*}
	\frac{\|\zeta\|_{L^2}^2}{\|\zeta\|^2}=
	\frac{\zeta^\top\Op_\infty \zeta}{\|\zeta\|^2}-
	\frac{\zeta^\top\big(\Op_\infty-\Op_T\big)\zeta}{\|\zeta\|^2} \ge \lambda_{\min}(\Op_\infty)-
	\frac{\zeta^\top\big(\Op_\infty-\Op_T\big)\zeta}{\|\zeta\|^2}.
\end{align*}
The proof of Lemma F.2 in \cite{basu2015} allows to deduce from \eqref{ineq:appl} that, for any $R \geq 2\slash \sqrt{T},$
\[
\PP\left(\sup_{\zeta\in\mathcal K(s),\|\zeta\|\le 1}\big|\zeta^\top\big(\Op_\infty-\Op_T\big)\zeta\big|>3R\right)\le 21^s \left(d \wedge \left(\frac{\e d}{s}\right)^s\right)  \exp\bigg(-\bigg(\frac{\sqrt{T}R}{\mathrm{e}\mathfrak{L}\mathfrak{W}} \bigg)^{\kappa(q,\eta)} \bigg).
\]
Thus,
\begin{align*}
	\PP\left(\inf_{\zeta\in\mathcal C(s,c_0)}\frac{\|\zeta\|_{L^2}^2}{\|\zeta\|^2}>\frac{\mathfrak e_\infty}{2}\right)
	&\ge \PP\left(\sup_{\zeta\in\mathcal C(s,c_0),\|\zeta\|\le1}\big|\zeta^\top\big(\Op_\infty-\Op_T\big)\zeta\big|\le \frac{\mathfrak e_\infty}{2}\right)\\
	&\ge \PP\left(\sup_{\zeta\in\mathcal K(2s), \|\zeta\|\le1}\big|\zeta^\top\big(\Op_\infty-\Op_T\big)\zeta\big|\le \frac{\mathfrak e_\infty}{6(c_0+2)}\right)\\
	&\ge 1-21^{2s}\left(d\wedge \left(\frac{\e d}{2s}\right)^{2s}\right)
	\exp\bigg(-\bigg(\frac{\sqrt{T}\mathfrak e_\infty}{18(c_0+2) \mathrm{e}\mathfrak{L}\mathfrak{W}} \bigg)^{\kappa(q,\eta)} \bigg),		
\end{align*}
resulting in the asserted condition on the sample size $T$.
For proving part (b), note first that the relation
\begin{align*}
\left\{\max_{i=1,\ldots,N}\big|\overline\psi_{ii,T}-\overline\psi_{ii,\infty}\big|>\frac{\mathfrak e_\infty}{2}\right\} \subset \bigg\{\sup_{\zeta\in\mathcal C(s,c_0)}\frac{\big|\zeta^\top\big(\Op_T-\Op_\infty\big)\zeta\big|}{\|\zeta\|^2}>\frac{\mathfrak e_\infty}{2}\bigg\}
\end{align*}
in particular implies that, for $T\ge T_0(\tfrac{\ep_0}{3},s,c_0,\mathfrak{LW})$,
\[
\PP\left(\max_{i=1,\ldots,N}\overline\psi_{ii,T} > \mathfrak D_\infty+\frac{\mathfrak{e}_\infty}{2}\right)\le \frac{\ep_0}{3}.
\]
It remains to control the deviation of the martingale term.
Given $\theta,\eta\in\R^N$, we write
\[
\frac2T\int_0^T\left(\sigma_0^{-1}\left(b_\theta-b_\eta\right)\right)^\top(X_t)\d W_t = 2(\theta-\eta)^\top\left(\ep_{1,T}, \ldots,\ep_{N,T}\right)^\top,
\]
where
\begin{equation}\label{eq:ep}
\ep_{i,T}\coloneqq\frac1T\int_0^T\left(\sigma_0^{-1}\psi_i\right)^\top(X_s)\d W_s,\quad i=1,\ldots,N.
\end{equation}
Note that the quadratic variation of this continuous martingale is given by
\begin{align*}
	\left\langle\ep_i\right\rangle_T&=\frac{1}{T^2}\int_0^T\left(\sigma_0^{-1}\psi_i\right)^\top(X_s)\left(\sigma_0^{-1}\psi_i\right)(X_s)\d s\\
	&=\frac{1}{T^2}\int_0^T\left\langle \psi_i,a_0^{-1}\psi_i\right\rangle(X_s)\d s=\frac1T\overline\psi_{ii,T}.
\end{align*}
Using Bernstein's inequality for continuous martingales and taking into account the condition on $\lambda$, we thus obtain for for $T\geq T_0(\tfrac{\ep_0}{3},s,c_0,\mathfrak{LW})$
\begin{align*}
	&\PP\left(\sup_{\theta\ne\eta\in\R^N}
	\frac{\frac2T\int_0^T\left(\sigma_0^{-1}\left(b_\theta-b_\eta\right) \right)^\top(X_t) \d W_t }{\|\theta-\eta\|_1}>\lambda\right)\\
	&\quad\le\PP\left(\sup_{\theta\ne\eta}\frac{2(\theta-\eta)^\top\left(\ep_{1,T},\ldots,\ep_{N,T}\right)}{\|\theta-\eta\|_1}>\lambda, \max_{i=1,\ldots,N}\overline\psi_{ii,T}<\mathfrak D_\infty+\frac{\mathfrak e_\infty}{2}\right)\\
	&\qquad+ \PP\left(\max_{i=1,\ldots,N}\overline\psi_{ii,T} > \mathfrak D_\infty+\frac{\mathfrak{e}_\infty}{2}\right)\\
	&\quad\le
	\sum_{i=1}^N\PP\left(2\lvert\ep_{i,T}\rvert>\lambda,\left\langle2\ep_i\right\rangle_T\le \frac4T\left(\mathfrak D_\infty+\frac{\mathfrak e_\infty}{2}\right)\right)+ \frac{\ep_0}{3}\\
	&\quad \le
	2N\exp\left(-\frac{T\lambda^2}{8\mathfrak D_\infty+4\mathfrak e_\infty}\right) + \frac{\ep_0}{3}\leq\frac{2\ep_0}{3}.
\end{align*}
\end{proof}

On the basis of the given key deviation inequalities, the machinery of high-dimensional statistics now allows the derivation of oracle inequalities. 
Our proof strategy follows the one developed in \cite{ciolek20}.

\begin{theorem}[oracle inequality]\label{thm:oracle}
Assume that we are given a continuous record of observations of the solution of \eqref{eq:b0}, where $b_0=b_{\theta_0}\in\mathcal V$ with $\|\theta_0\|_0\le s_0$.
Fix $\gamma>0$ and $\ep_0\in(0,1)$, and consider the Lasso estimator $\hat\theta_T$ introduced in \eqref{lasse}.
Then, for 
\begin{equation}\label{eq:lambda}
	\lambda\ge 2\sqrt{\frac{\left(2\mathfrak D_\infty+\mathfrak e_\infty\right)}{T}\cdot\log\left(\frac{6N}{\ep_0}\right)} \quad\text{ and }\quad T\ge T_0\left(\tfrac{\ep_0}{3},s_0,1+\tfrac2\gamma,\mathfrak{LW}\right),
\end{equation}
with probability at least $1-\ep_0$, we have
\begin{equation}\label{orac:ineq}
\big\|\hat\theta_T-\theta_0\big\|_{L^2}^2
\le (1+\gamma)\inf_{\theta\in\R^N:\|\theta\|_0\le s_0} \left\{\big\|\theta-\theta_0\big\|_{L^2}^2+\ \frac{4(2+\gamma)^2}{\gamma(1+\gamma)\mathfrak e_\infty}\ \|\theta\|_0\lambda^2\right\}.
\end{equation}
Furthermore, for any $\lambda$ fulfilling \eqref{eq:lambda} and $T\ge T_0\left(\tfrac{\ep_0}{3},s_0,1,\mathfrak{LW}\right)$, with probability at least $1-\ep_0$, 
\begin{equation}\label{eq:uppbou}
\big\|\hat\theta_T-\theta_0\big\|_{L^2}^2\le \lambda^2\cdot\frac{2s_0}{\mathfrak e_\infty}.
\end{equation}
\end{theorem}
By specifying $\lambda$ as proposed in \eqref{eq:lambda}, the previous result implies that an upper bound of order $(s_0\log N)/T$ on the squared $L^2$ risk of the Lasso estimator $\hat \theta_T$ holds with high probability.
In particular, in terms of the rate of convergence, our techniques give the same results as the concentration inequalities tailored to the specific OU model used in \cite{gaiffas19} and \cite{ciolek20}, respectively.

\begin{proof}[Proof of Theorem \ref{thm:oracle}]
Recall the definition of the event $\mathscr E(s,c_0)$ according to \eqref{def:eve}, and let $s\ge s_0$.
On $\mathscr E(s,c_0)$, the basic inequality \eqref{eq:bas1} implies for any $\theta\in\R^N$ that
\begin{align}\nonumber
\big\|\hat\theta_T-\theta_0\big\|_{L^2}^2+\lambda\big\|\hat\theta_T-\theta\big\|_1
&\le \big\|\theta-\theta_0\big\|_{L^2}^2+\lambda \left(\big\|\theta\big\|_1 -\big\|\hat\theta_T\big\|_1+\big\|\hat\theta_T-\theta\big\|_1\right)\\\label{basineq}
&\le \big\|\theta-\theta_0\big\|_{L^2}^2+2\lambda\big\|\hat\theta_T|_{\operatorname{supp}(\theta)}-\theta\big\|_1.
\end{align}
Assume now that $\theta\in\R^N$ fulfills $\|\theta\|_0\le s_0$, and consider the event
\begin{equation}\label{eqevent}
2\lambda \big\|\hat\theta_T|_{\operatorname{supp}(\theta)} - \theta\big\|_1>\gamma\big\|\theta-\theta_0\big\|_{L^2}^2.
\end{equation}
If this does not occur, \eqref{orac:ineq} holds true since we obtain from \eqref{basineq} that
\[
\big\|\hat\theta_T-\theta_0\big\|_{L^2}^2\le(1+\gamma)\big\|\theta-\theta_0\big\|_{L^2}^2.
\]
Otherwise, if \eqref{eqevent} holds, we have on $\mathscr E(s,c_0)$ that
\begin{align*}
\lambda\big\|\hat\theta_T-\theta\big\|_1
&\le \big\|\theta-\theta_0\big\|_{L^2}^2+2\lambda\big\|\hat\theta_T|_{\operatorname{supp}(\theta)}-\theta\big\|_1\\
&\le
\frac{2\lambda}{\gamma}\big\|\hat\theta_T|_{\operatorname{supp}(\theta)}-\theta\big\|_1
+2\lambda\big\|\hat\theta_T|_{\operatorname{supp}(\theta)}-\theta\big\|_1
\end{align*}
i.e., $\hat\theta_T-\theta\in\mathcal C(s,c_0)$ for the choice $c_0=1+\tfrac2\gamma$, such that, after using Cauchy--Schwarz, 
\[
\big\|\hat\theta_T|_{\operatorname{supp}(\theta)}-\theta\big\|_1
\le \big\|\hat\theta_T-\theta\big\|\cdot{\sqrt{\|\theta\|_0}}\le \sqrt{\frac{2\|\theta\|_0}{\mathfrak e_\infty}}\big\|\hat \theta_T-\theta\big\|_{L^2}.
\]
Summing up,
\[
\big\|\hat\theta_T-\theta_0\big\|_{L^2}^2\le \big\|\theta-\theta_0\big\|_{L^2}^2+ 2\lambda\sqrt{\frac{2\|\theta\|_0}{\mathfrak e_\infty}} \left(\big\|\hat \theta_T-\theta_0\big\|_{L^2}+ \big\|\theta-\theta_0\big\|_{L^2}\right).
\]	
Applying the Young inequalities
\[
\big\|\hat\theta_T-\theta_0\big\|_{L^2}\le \frac{ax}{2}+	\big\|\hat\theta_T-\theta_0\big\|_{L^2}^2\cdot\frac{1}{2ax}, \quad
\big\|\theta-\theta_0\big\|_{L^2}\le \frac{ax}{2}+	\big\|\theta-\theta_0\big\|_{L^2}^2\cdot\frac{1}{2ax},
\]
with $a=(\gamma+2)/\gamma$ and $x=\lambda\sqrt{2\|\theta\|_0/\mathfrak e_\infty}$, we finally obtain
\[
\big\|\hat\theta_T-\theta_0\big\|_{L^2}^2\le(1+\gamma)\left( \big\|\theta-\theta_0\big\|_{L^2}^2 + \frac{4(2+\gamma)^2}{\gamma(1+\gamma)\mathfrak e_\infty}\ \lambda^2\|\theta\|_0\right).
\]
For the proof of \eqref{eq:uppbou}, note that, taking $\theta=\theta_0$, \eqref{basineq} implies that, on $\mathscr E(s_0,c_0)$,
\[
\big\|\hat\theta_T-\theta_0\big\|_{L^2}^2+\lambda\big\|\hat\theta_T-\theta_0\big\|_1
\le 2\lambda\big\|\hat\theta_T|_{\operatorname{supp}(\theta_0)}-\theta_0\big\|_1.
\]
Now, since $\hat\theta_T-\theta_0\in\mathcal C(s_0,1)$ on $\mathscr E(s_0,c_0)$,
\[
\big\|\hat\theta_T-\theta_0\big\|_{L^2}^2\le 
\lambda\big\|\hat\theta_T|_{\operatorname{supp}(\theta_0)}-\theta_0\big\|_1
\le \lambda\sqrt{\frac{2s_0}{\mathfrak e_\infty}}\big\|\hat\theta_T-\theta_0\big\|_{L^2},
\]
which already gives the asserted inequality.
\end{proof}

\begin{remark}	
	While our results are non-asymptotic, we do face a restriction in that the constant $\mathfrak W$ appearing in the lower bound for the required sample size (see \eqref{eq:lambda}) is not explicit. However, it appears to be very demanding to work out explicit constants in a general framework. 
	In the spirit of \cite{pokern2009}, our arguments could also be carried out for the more restricted class of \emph{reversible} diffusion processes by assuming a parametric form of the potential function and then considering a parametrized drift function $b_\theta(x) = \frac{1}{2}\operatorname{div}(a_0(x))-\frac{1}{2}a_0(x)\nabla V_\theta(x)$ for $V_\theta\in\mathcal V$. 
	Although functional inequalities (e.g., of Poincaré-type) are applicable in this reversible framework, the control of the constants involved still constitutes a fundamental challenge. 
\end{remark}

We conclude this section by briefly categorising the results and sketching potential future research. 
Note first that Theorem 2.7 in \cite{dex22} provides a lower bound on the Frobenius norm for the estimation of the matrix $\bm A$ in the $d$-dimensional OU model \eqref{gen:ou} with $\sigma=\mathbb{I}_{d}$ over the class of $s_0$-sparse matrices. 
Translating the number of parameters into our framework, the lower bound is of order $s_0\log(N/s_0)/T$.
Compared to the upper bound of order $(s_0\log N)/T$, there is thus only a logarithmic gap, appearing in this very form also in \cite{gaiffas19} and \cite{ciolek20}.
As demonstrated in \cite{dex22} in the context of drift estimation for Lévy-driven OU processes, the key to eliminating the logarithmic gap lies in a refined deviation inequality for the stochastic error (in our context specified as $\ep_{i,T}$ as defined in \eqref{eq:ep}).
In fact, the combination of concentration inequalities in the sense of Lemma \ref{lem:conc} (which is a rather straightforward consequence of our general Theorem \ref{prop:conc}) with the techniques from Section 3.2 of \cite{dex22} can be expected to allow the derivation of minimax optimal penalized estimators also for general diffusion models.

\subsection{MCMC for moderately heavy-tailed targets} \label{sec:mcmc}
In general, Markov chain Monte Carlo (MCMC) is a collective term for algorithms relying on ergodicity of Markov chains that are (i) easy to simulate and (ii) specifically designed such that their invariant distribution approximates a given target density, for which samples are to be obtained. 
These algorithms have a long and rich history.
At this point, we cannot give a detailed account of the literature which would do justice to the field, but only want to point out its fundamental importance in connected areas such as Bayesian optimization or inverse problems in high dimensional contexts, where the posterior distribution becomes the target. Other than the fundamental theoretical work in \cite{dala17,durmus17} that will be discussed below, we refer to \cite{dala19, durmus19, durmus19b, erdogdu18,erdogdu21,teh16,vollmer16} for some recent contributions that motivated our study. Our particular interest lies on the so called \textit{Unadjusted Langevin Algorithm} (ULA), which we describe next.

Suppose that we are given a target density $\pi \propto \exp(-U(x))$ for some continuously differentiable function $U\colon \R^d \to \R$, which is usually referred to as the \textit{potential}. Let us also assume that $\nabla U$ is $L$-Lipschitz continuous such that the (unadjusted or overdamped) \textit{Langevin diffusion}
\[\diff{X}_t = -\nabla U(X_t) \diff{t} + \sqrt{2} \diff{W_t}, \quad t \geq 0,\]
has a unique strong solution, which is a Feller Markov process with invariant distribution 
\[\pi(\diff{x}) = \frac{1}{\int_{\R^d} \exp(-U(y)) \diff{y}} \exp(-U(x)) \diff{x}, \quad x \in \R^d.\]
To obtain samples with approximate distribution $\pi$ and to approximate integrals $\pi(f)$ for $\pi$-integrable functions $f$ via the corresponding Monte Carlo estimator, in practice one needs to discretize the SDE to make simulation procedures feasible. The ULA uses a simple Euler discretization scheme as numerical SDE approximation, where the Euler discretization with step size $\Delta$ is the Markov chain given by the stochastic difference equation 
\[\vartheta^{(\Delta)}_{n+1} = \vartheta^{(\Delta)}_{n} - \Delta\nabla U(\vartheta^{(\Delta)}_{n}) + \sqrt{2\Delta} \xi_{n+1}, \quad n \in \N_0, \quad \vartheta^{(\Delta)}_0 \overset{\mathrm{d}}{=} X_0,\]
where $(\xi_n)_{n \in \N}$ is a sequence of i.i.d.\ standard normal random variables on $\R^d$ independent of $\vartheta^{(\Delta)}_0$. Sampling such a chain is computationally efficient, provided the gradient $\nabla U$ can be cheaply evaluated.
By considering the time-inhomogeneous Markov process given as the strong solution to the SDE
\[\diff{Z^{(\Delta)}_t} = b(\mathbf{Z}^{(\Delta)},t) \diff{t} + \sqrt{2} \diff{W_t}, \quad t \geq 0, \quad Z^{(\Delta)}_0 = X_0,\]
with non-anticipatory drift $b(\bm{z},t) = -\sum_{k=0}^\infty \nabla U(z_{k\Delta})\one_{[k\Delta,(k+1)\Delta)}(t)$ for $(\bm{z},t) \in \mathcal{C}(\R_+,\R^d) \times \R_+$, it is straightforward to show that the laws of $(\vartheta^{(\Delta)}_{n})_{n \in \N_0}$ and  $(Z^{(\Delta)}_{n\Delta})_{n \in \N_0}$ coincide. 

It has been observed in the literature \cite{dala17,durmus17} that for potentials $U$ that are either strongly convex---i.e., $\pi$ is strongly log-concave---or that are convex and superexponential outside some ball, explicit requirements on the step length $\Delta$ and sample size $n$ can be formulated to guarantee sampling with $\varepsilon$-precision in total variation or Wasserstein distance. 
For strongly log-concave densities, the natural connection to the gradient descent in a convex setting is pointed out in \cite{dala17}. 

We now apply our previous results to obtain PAC bounds and related suggestions for sample size $n$, burn-in $m$ and discretization step $\Delta$  for the ULA Monte Carlo estimator of polynomially growing functions for moderately heavy-tailed target densities $\pi$ such that \[\exists \iota > 0, q \in (0,1)\quad\text{such that}\quad \int_{\R^d}\exp(\iota \lVert x \rVert^{1-q}) \, \pi(\diff{x}) < \infty.
\]
As follows from \eqref{eq:subexp_mom2}, this is the case if $-\nabla U$ satisfies \ref{cond:drift} with $q \in (0,1)$, i.e., there exists some $M_0,\mathfrak{r} > 0$ such that 
\begin{enumerate}[label = ($\mathscr{U}(q)$), ref = ($\mathscr{U}(q)$),leftmargin=*]
\item $\langle \nabla U(x), x/\lVert x \rVert \rangle \geq \mathfrak{r}\lVert x \rVert^{-q}, \quad \lVert x \rVert \geq M_0.$ \label{cond:gradient}
\end{enumerate}
This setting differs substantially from the (strongly) convex setting in \cite{dala17,durmus17}, whose assumptions imply \ref{cond:gradient} with $q \in [-1,0)$ and therefore, in particular, require the targets to have exponential moments, i.e., light tails. Heavy-tailed target densities implied by our assumption $q\in (0,1)$ become relevant, e.g., in Bayesian inverse problems with heavy-tailed noise or prior.   
As the following result demonstrates, the Euler discretized Markov chain $\bm{\vartheta}^{(\Delta)}$ under \ref{cond:gradient} inherits the subexponential ergodic behaviour from the original Langevin diffusion $\X$, provided that $U$ does not grow too fast. 
The proof is a straightforward application of the results from \cite{douc04}---which is the discrete-time counterpart to \cite{douc2009}---and is postponed to Appendix \ref{app:conv}. Let $(\mathbf{P}^x)_{x \in \R^d}$ be a family of probability measures such that $\bm{\vartheta}^{(\Delta)}$ is started in $x$ under $\mathbf{\P}^x$.

\begin{proposition}\label{prop:euler_conv}
Let $q \in (0,1)$. 
Suppose that $U$ satisfies \ref{cond:gradient} and, moreover, for some $M_1 > 0$, $\lVert \nabla U(x) \rVert \leq \lVert x \rVert^{\beta}$ for $\beta \leq (1-q)/2$. Then, for any $\Delta > 0$ in case $\beta < (1-q)/2$ or any $\Delta \leq \mathfrak{r}$ in case $\beta = (1-q)/2$, there exists an invariant probability measure $\pi_{(\Delta)}$ for the chain $\bm{\vartheta}^{(\Delta)}$ and there are constants $c = c(q,\Delta) > 0$ and $\tilde{c} = \tilde{c}(q,\Delta)$ such that, for $f_{q}(x) \sim (1+\lVert x \rVert)^{-2q} \exp(\tilde{c} \lVert x \rVert^{1-q})$ and $r_q(n) \sim n^{-2q/(1+q)} \exp(c n^{(1-q)/(1+q)})$, we have for any $x \in \R^d$ and pairs of inverse Young functions $(\Psi_1,\Psi_2) \in \mathcal{I}$
\[\lim_{n \to \infty} \Psi_1(r_q(n)) \big \lVert \mathbf{P}^x(\vartheta^{(\Delta)}_n \in \cdot) - \pi_{(\Delta)} \big\rVert_{\Psi_2 \circ f_q} = 0.\]
\end{proposition}

This convergence behaviour is in line and in fact states more precisely the findings from \cite[Section 3]{roberts96} for the ULA in $d=1$ for the model class of sub-Weibull distributions. It is vital to note that $\pi$ and $\pi_{(\Delta)}$ do not coincide, so even if the ULA converges at subgeometric rate for fixed step size $\Delta$, we need to choose $\Delta$ appropriately small to obtain useful approximations. We make this tuning parameter choice precise in the following.

Typical potentials satisfying \ref{cond:gradient}---such as $U(x) \propto \lVert x \rVert^{1-q}$ outside some ball centered around $0$---are not convex, and their gradient may converge at infinity. 
In fact, if we have $\lim_{\lVert x \rVert \to \infty} \lVert \nabla U(x) \rVert = 0$, then \cite[Theorem 2.4]{roberts96} implies that the Langevin diffusion $\X$ is not exponentially ergodic.
Hence, we cannot expect $\pi$ to have exponentially decaying tails.  Therefore, in contrast to the usually encountered potentials exhibiting some degree of convexity, it is quite natural for our purposes to assume that $\nabla U$ is bounded under \ref{cond:gradient} for $q \in (0,1)$. This makes it easy to prove the following result quantifying convergence of ULA to the target $\pi$ and the performance of the ULA Monte Carlo estimator with burn-in $m$, 
\[\mathbb{H}^{\bm{\vartheta}^{(\Delta)}}_{m,n,\Delta}(f) \coloneqq \frac{1}{n} \sum_{k=m+1}^{m+n} f\big(\vartheta^{(\Delta)}_{k\Delta}\big),\]
based on our results from Section \ref{sec:discrete} and the Girsanov argument underlying the total variation convergence result from \cite{dala17} for strongly convex potentials. 
Denote by $\PP_{\X}^{x,n\Delta}$ and $\PP^{x,n\Delta}_{\mathbf{Z}^{(\Delta)}}$ the laws of $(X_t)_{t \in [0,\Delta n]}$ and $(Z^{(\Delta)}_t)_{t \in [0,n\Delta]}$, respectively, under $\PP^x$.

\begin{proposition} 
Suppose that $U \in \mathcal{C}^1(\R^d)$ has an $L$-Lipschitz continuous and bounded gradient that satisfies \ref{cond:gradient} for some $q \in (0,1)$.
\begin{enumerate}[label = (\roman*), ref =(\roman*), leftmargin = *]
\item For any $\Delta \in (0,1]$ and initial distribution $\nu$ such that $V_q \in L^1(\nu)$, it holds for any $n \in \N$, 
\begin{equation}
\begin{split}\label{eq:tv_ula}
\big\lVert \mathbf{P}^\nu\big(\vartheta^{(\Delta)}_n \in \cdot\big) - \pi \big\rVert_{\mathrm{TV}}&\leq \mathfrak{C} \nu(V_q) \exp\big(-(\iota^{\prime\prime} n\Delta)^{(1-q)/(1+q)} \big)\\
&\hspace*{5em}+ \sqrt{\frac{(1 + \lVert \lVert \nabla U(\cdot) \rVert^2_\infty \rVert_{\infty}) dL^2n\Delta^2}{2}},
\end{split}\end{equation}
for some constant $\mathfrak{C} > 0$ and $\iota^{\prime \prime} \in (0,\iota^{(1+q)/(1-q)}(1+q)( \mathfrak{r} - \iota(1-q)))$ for some $\iota < \mathfrak{r}/(1-q)$.
\item Let $\eta_1,\eta_2,\eta_3 \geq 0$,  $f \in \mathcal{G}(\eta_1, \mathfrak{L}) \cap \mathcal{W}^{2,p}_{\mathrm{loc}}(\R^d)$, $p \geq d$, with $\nabla f \in L_{\mathrm{loc}}^{2d}(\R^d)$ such that $\lVert \nabla f(x) \rVert \lesssim 1 + \lVert x \rVert^{\eta_2}$ and for all $i,j=1,\ldots,d$, $\lvert \uppartial_{x_i,x_j} f(x) \rvert \lesssim 1 + \lVert x \rVert^{\eta_3}$. Let also $C,\iota^{\prime\prime},\alpha,\tilde{\gamma}, \tilde\varsigma = \tilde\varsigma(\eta_1,q,0), \varrho = \varrho(\alpha,\eta_2,\tilde{\gamma},q)$ be the constants from Corollary \ref{coro:pac_discrete}, adapted to the specific parameters of the Langevin diffusion. 
Then, for $\Delta$ satisfying both $\Delta < \min\{1, \varepsilon/(3\mathrm{e}\mathfrak{D}), (\log(1/\delta))^{\tilde\varsigma-\varrho}\}$ and
\begin{align*}
&\Delta \leq  \frac{(\delta \varepsilon)^2}{2(1 + \lVert \lVert \nabla U(\cdot) \rVert_\infty \rVert_{\infty}) dL^2 \big((\log(4/\delta))^{2\tilde\varsigma} + \varepsilon^2\big(2 + (\log(4C/\delta))^{(1+q)/(1-q)}\slash \iota^{\prime\prime} \big)\big)},
\end{align*}
sample size 
\[n = n(\Delta,\varepsilon,\delta) = \big\lceil \Psi(\Delta,\varepsilon,\delta/4) \big\rceil\] 
and burn-in 
\[m = m(\Delta,\varepsilon,\delta) = \big\lceil \Delta^{-1} (\log(4C/\delta))^{(1+q)/(1-q)}\slash\iota^{\prime\prime}\big\rceil,\]
it holds for any initial distribution $\nu$ such that $V_q \in L^1(\nu)$ that 
\[\mathbf{P}^\nu\big(\big\lvert \mathbb{H}^{\bm{\vartheta}^{(\Delta)}}_{m,n,\Delta}(f) - \pi(f)\big\rvert \leq \varepsilon \big) \geq 1- \delta.\] \label{theo:ula2}
\end{enumerate}
\end{proposition}
\begin{proof}
As in the proof of Lemma 2 in \cite{dala17}, see also \cite{dala12}, using $L$-Lipschitz continuity of $\nabla U$ and Girsanov's theorem, it follows that the Kullback--Leibler divergence of $\PP_{\X}^{x,n\Delta}$ wrt $\PP^{x,n\Delta}_{\mathbf{Z}^{(\Delta)}}$
fulfills
\begin{equation*}\label{eq:kl}
\mathrm{KL}\big(\PP^{x,n\Delta}_{\X} \big\Vert \PP^{x,n\Delta}_{\mathbf{Z}^{(\Delta)}} \big) \leq \frac{L^2 \Delta^3}{12} \sum_{k=0}^{n-1} \E^x\big[\big \lVert \nabla U(Z^{(\Delta)}_{k \Delta})\big\rVert^2 \big] + \frac{dL^2n\Delta^2}{4}.
\end{equation*}
Thus, using $\lVert \lVert \nabla U (\cdot) \rVert\rVert_\infty \leq \sqrt{d} \lVert \lVert \nabla U (\cdot)\rVert_\infty \rVert_\infty $ and Pinsker's inequality, it follows that 
\begin{equation}\label{eq:tv_pinsker}
\big\lVert \PP^x(X_{n \Delta} \in \cdot) - \mathbf{P}^x\big(\vartheta^{(\Delta)}_n \in \cdot \big) \big\rVert_{\mathrm{TV}} \leq \big\lVert \PP^{x,n\Delta}_{\X} -\PP^{x,n\Delta}_{\mathbf{Z}^{(\Delta)}} \big\rVert_{\mathrm{TV}} \leq \sqrt{\frac{(1 + \lVert \lVert \nabla U(\cdot) \rVert^2_\infty\rVert_{\infty}) dL^2n\Delta^2}{2}}.
\end{equation}
By triangle inequality, subexponential convergence in \eqref{eq:subexp_tv} with the parameters adapted to the Langevin diffusion and \eqref{eq:tv_pinsker}, we immediately obtain \eqref{eq:tv_ula}. Moreover, for $\Delta$ given as in part \ref{theo:ula2}, the choice $n = n(\Delta,\varepsilon,\delta)$ and $ m= m(\Delta,\varepsilon,\delta)$, if we define 
\[g_{\varepsilon}((x_t)_{t \in [0,(n+m)\Delta]}) = \one_{(\varepsilon, \infty)}\Big(\Big\vert \frac{1}{n} \sum_{k=m+1}^{n+m}  (f(x_{k\Delta}) - \pi(f))\Big\vert\Big), \quad (x_t)_{t \in [0,(n+m)\Delta]} \in \mathcal{C}([0,(n+m)\Delta], \R^d),\]
it follows from \eqref{eq:tv_pinsker}  that 
\begin{align*}
&\lvert \PP^\nu(\lvert \mathbb{H}_{m,n,\Delta}(f) - \pi(f) \rvert > \varepsilon) - \mathbf{P}^\nu(\lvert \mathbb{H}^{\bm{\vartheta}^{(\Delta)}}_{n,m,\Delta}(f) - \pi(f) \rvert > \varepsilon)\rvert\\
&\quad= \big\lvert \E^\nu\big[g_{\varepsilon}\big((X_t)_{t \in [0,(n+m)\Delta]}\big) \big] - \E^\nu\big[g_{\varepsilon}\big((Z^{(\Delta)}_t)_{t \in [0,(n+m)\Delta]}\big) \big] \big\rvert\\
&\quad\leq \int_{\R^d} \big\lVert \PP^{x,(n+m)\Delta}_{\X} -\PP^{x,(n+m)\Delta}_{\mathbf{Z}^{(\Delta)}} \big\rVert_{\mathrm{TV}} \, \nu(\diff{x})\\
&\quad\leq \sqrt{\frac{(1 + \lVert \lVert \nabla U(\cdot) \rVert^2_\infty \rVert_{\infty}) dL^2(n+m)\Delta^2}{2}}\\
&\quad\leq \bigg(\frac{(1 + \lVert \lVert \nabla U(\cdot) \rVert^2_\infty \rVert_{\infty}) dL^2 \Delta}{2} \Big(2+ \frac{(\log(4/\delta))^{2\tilde\varsigma}}{\varepsilon^2} + (\log(4C/\delta))^{(1+q)/(1-q)}\slash \iota^{\prime\prime} \Big)\bigg)^{1/2}\\
&\quad\leq \delta/2.
\end{align*}
Statement \ref{theo:ula2} now follows from Corollary \ref{coro:pac_discrete} and triangle inequality.
\end{proof}

The above result gives lower bounds on the required step length, sample size and burn-in for  an $\varepsilon$-precise integral approximation of $\pi(f)$ with probability at least $1-\delta$ for polynomially bounded $f$ with polynomially bounded weak derivative and Hessian. These are summarized in Table \ref{tab:mcmc}. An obvious application of this result are explicit finite sample guarantees for MCMC moment approximations of the target $\pi$.
\begin{table}[h]
\centering
\begin{tabular}{l | c | c | c}
& step length $\Delta$ & sample size $n$ & burn-in $m$ \\\hline
$\varepsilon$-prec.\ sampling & $\phantom{\bigg(}\frac{\varepsilon^2}{d(\log(\mathfrak{C}/\varepsilon))^{(1-q)/(1+q)})} $ & $\frac{d(\log(\mathfrak{C}/\varepsilon))^{2(1-q)/(1+q)}}{\varepsilon^2}$ & $-$\\ \hline 
$(\varepsilon,\delta)$-PAC bound & $\phantom{\bigg(}\frac{(\delta \varepsilon)^2}{d( \log(1/\delta))^{2(\eta_1 + (q+3)/2)/(1-q)}}$ & $\frac{d\mathfrak{D}^2 (\log(1/\delta))^{(4(\eta_1 + (q+3)/2))/(1-q)}}{\delta^2 \varepsilon^4}$ & $\frac{d (\log(1/\delta))^{2(\eta_1 + q + 2)/(1-q)}}{(\delta \varepsilon)^2}$
\end{tabular}
\caption{Order of sufficient sampling frequency $\Delta$, sample size $n$ and burn-in $m$ for $(\varepsilon,\delta)$-PAC bounds and sampling within $\varepsilon$-TV margin}
\label{tab:mcmc}
\end{table}

\begin{remark}
It should be noted that the exact dimensional dependence of $\mathfrak{D}$ is not clear, which, similarly to the previous section, is an effect of unspecified constants in the ergodicity and Sobolev bounds used for the derivation of the concentration inequalities. Overcoming this issue is highly non-trivial and subject of ongoing research efforts. In contrast, the convex, respectively strongly convex, settings in \cite{durmus17,dala17} give rise to Poincaré, respectively log-Sobolev, inequalities with explicit constants such that the investigated required number of iterations for sampling within an $\varepsilon$-margin in total variation can be made explicit in terms of the dimension in these papers. According to the above, the simulation grid should be significantly finer and the sample size significantly larger to obtain exact PAC-guarantees compared to the case when one would just be interested in sampling with $\varepsilon$-precision in total variation. Here, the dependence on the level $\varepsilon$ is a natural correspondence to the sample sizes (and hence necessary number of gradient evaluations) found in  \cite{dala17,durmus17}. 
\end{remark}

Our results yield explicit and useful guarantees for a sampling scenario that is quite different from what is usually encountered in the theoretical MCMC literature. Still, we expect that the dependence of $(\Delta,n,m)$ on $\delta$ for the PAC bounds can be improved in the sense that the $\delta^2$-dependency is likely too strict.  Its occurrence is explained by our strategy to control the total variation distance between the law of the Langevin diffusion $\X$ and its numerical approximation $\mathbf{Z}^{(\Delta)}$ in terms of their KL-divergence using Pinsker's inequality. 
This leads to a suboptimal bound on the total variation distance, causing the additional dependence on $\delta^2$. We are not aware of any other approaches in the MCMC literature to control this loss on the path level. This issue can be possibly circumvented by deriving concentration inequalities for $\mathbb{H}_{m,n,\Delta}^{\bm{\vartheta}^{(\Delta)}}(f)$ around its mean directly and lift these to concentration inequalities of $\mathbb{H}_{m,n,\Delta}^{\bm{\vartheta}^{(\Delta)}}(f)$ around the target $\pi(f)$ by establishing appropriate bias estimates. This is the strategy pursued in \cite[Proposition 18]{durmus15}---an earlier preprint version of \cite{durmus17}---where the authors 
infer a sufficient sample size $n \sim d\log(1/\delta)/\varepsilon^4$ and sampling frequency  $\Delta \sim \varepsilon^2/d$ for the ULA MC estimator of the integral $\pi(f)$ for a strictly log-concave density (in particular, $q = -1$) and bounded $f$. Since we focus on applications that can be treated with our theoretical results from Section \ref{sec:discrete}, we do not push further the issue of improving our bounds in the setting of a heavy-tailed target $\pi$ and unbounded integrands $f$. Instead, we leave it open as an interesting question for future research.

\begin{appendix} 
\section{Remaining proofs}\label{app:conv}
\begin{proof}[Proof of Proposition \ref{prop:f_conv}]
By \cite[Proposition 1]{douc2009}, every compact set is petite and any skeleton chain is irreducible. 
Moreover, if we let $V \in \mathcal{C}^2(\R^d)$ such that $V = \lVert x \rVert^\gamma$ for $\lVert x \rVert \geq M_0$ and $V \geq 1$,  and we can show that $LV$ is locally bounded and
\begin{equation}\label{eq:gen_drift}
L V(x) \lesssim -\phi \circ V(x) (1+ o(1)), \quad \lVert x \rVert \geq M_0,
\end{equation}
for $\phi(x) = \mathfrak{r}\gamma x^{(\gamma-(1+q))/\gamma}$ which is increasing, differentiable and concave on $(0,\infty)$, it will follow from \cite[Theorem 3.4]{douc2009} that, for any $\varepsilon \in(0,1)$, the condition $\mathbf{D}(C_{\varepsilon},V,\phi_\varepsilon,a_{\varepsilon})$ is satisfied for $\phi_\varepsilon = (1-\varepsilon)\phi$, $C_\varepsilon = \overbar{B(0,M_\varepsilon)}$ for $M_\varepsilon \geq M_0$ large enough and $a_\varepsilon = \sup_{\lVert x \rVert \leq M_\varepsilon} \lvert LV (x) + \phi_{\varepsilon} \circ V(x)\rvert$. 
This then implies the result using Theorem 3.2 and Proposition 4.6 from \cite{douc2009}. 
(Note that in the notation of \cite{douc2009}, $f^\ast = \phi_\varepsilon \circ V \sim f_{\gamma,q}$, $H_{\phi_\varepsilon}^{-1}(t) = (1+ (1+q)(1-\varepsilon) t/\gamma)^{\gamma/(1+q)}$ for $q \in (-1,1)$ and $H_{\phi_\varepsilon}^{-1}(t) = \exp(-\mathfrak{r}\gamma(1-\varepsilon)t)$ for $q = -1$, hence $r_\ast(t) = \phi_\varepsilon \circ H^{-1}_{\phi_\varepsilon}(t) \sim r_{\gamma,q}(t)$.) 
Since $b,\sigma$ are locally bounded and $L$ is a local operator, it is immediate that $LV$ is locally bounded as well. 
Further, for $\lVert x \rVert \geq M_0$, \ref{cond:drift} implies
\begin{align*} 
\langle b(x), \nabla V(x) \rangle &= \gamma \lVert x \rVert^{\gamma-1} \langle b(x), x \slash \lVert x \rVert \rangle \leq -r\gamma \lVert x \rVert^{\gamma-1-q} = -\phi \circ V(x),
\end{align*}
and the assumptions on the diffusion matrix yield
\begin{align*} 
\lvert \mathrm{tr}\big(a(x) D^2 V(x) \big)\lvert &= \Big\lvert \sum_{i,j = 1}^d a_{i,j}(x) \big(\one_{\{i=j\}}\gamma \lVert x \rVert^{\gamma - 2} +  \gamma(\gamma-2) x_i x_j \lVert x \rVert^{\gamma - 4}\big)\Big\rvert\\
&\leq (\Lambda \gamma d +\gamma\lvert \gamma-2 \rvert \lambda_+)  \lVert x \rVert^{\gamma-2} = o(\phi \circ V(x)).
\end{align*}
This gives \eqref{eq:gen_drift} and therefore the result.
\end{proof}

\begin{proof}[Proof of Proposition \ref{prop:euler_conv}]
Let $P^{(\Delta)}(x,B) = \mathbf{P}^x(\vartheta^{(\Delta)}_n \in B)$, $(x,B) \in \R^d \times \mathcal{B}(\R^d)$, be the transition kernel of the Markov chain $\bm{\vartheta}^{(\Delta)}$. 
Since $P^{(\Delta)}(x,\cdot) = \mathcal{N}(x-h \nabla U(x), 2h \mathbb{I}_d)$, it follows from classical Meyn--Tweedie arguments (cf.\ \cite[Theorem 3.1]{hansen03} for the precise statement) that $P^{(\Delta)}$ is an aperiodic and $\lebesgue$-irreducible Markov kernel and that all compact sets are small and hence petite. 
Let $\Phi(x) = x - \Delta \nabla U(x)$ such that we may write $\vartheta_{n+1}^{(\Delta)} = \Phi(\vartheta^{(\Delta)}_n) + \sqrt{2}\Delta \xi_{n+1}$. By our assumptions on the gradient $\nabla U$, we can choose $M \geq M_0 \vee M_1 \vee 1$ large enough such that, for $\lVert x \rVert \geq M$, we have 
\begin{align*}
\lVert \Phi(x) \rVert^2 &= \lVert x \rVert^2 - 2\Delta \langle x, \nabla U(x) \rangle + \Delta^2 \lVert \nabla U(x) \rVert^2\\
&\leq \lVert x \rVert^2 - 2\Delta \mathfrak{r}\lVert x \rVert^{1-q} + \Delta^2 \lVert x \rVert^{2\beta}\\
&\leq \lVert x \rVert^2\big(1- \Delta \mathfrak{r}\lVert x \rVert^{-(1+q)}\big)\\
&\leq \big(\lVert x \rVert \big(1- \tfrac{\mathfrak{r}\Delta}{2} \lVert x \rVert^{-(1+q)}\big)\big)^2.
\end{align*}
Hence, Assumption 3.4 from \cite{douc04} is fulfilled with $R_0 = M, \rho = 1+q, r = \mathfrak{r}\Delta/2$. Moreover, since the noise $(\xi_n)_{n \in \N}$ is i.i.d.\ Gaussian, Assumption 3.3 from \cite{douc04} is satisfied for any $z_0 > 0$ and $\gamma_0=1$. It thus follows from \cite[Theorem 3.3]{douc04} that their central drift condition $\mathbf{D}(\phi,V,C)$ holds for $\phi(v) = cv(1+ \log v)^{-2q/(1-q)}$, $V(x) = \mathrm{e}^{z\lVert x \rVert^{1-q}}$ and the compact set $C = \overbar{B(0, \tilde{M})}$, for some $c,z > 0$ and $\tilde{M} \geq M$. Consequently, for $H_\phi(t) = \int_1^t 1/\phi(v)\diff{v}$, we have $r_q \sim \phi \circ H_{\phi}^{-1}$ and $f_q \sim \phi \circ V$ for appropriate choices of the constants $c(q,\Delta), \tilde{c}(q,\Delta)$. Since $P^{(\Delta)}$ is irreducible and aperiodic and $C$ is petite, we can now apply \cite[Theorem 2.8]{douc04} to prove the claim.
\end{proof}

\end{appendix}

\paragraph{Acknowledgements}
CS and LT gratefully acknowledge financial support of Carlsberg Foundation Young Researcher Fellowship grant CF20-0640 ``Exploring the potential of nonparametric modelling of complex systems via SPDEs''.

\printbibliography
\end{document}